\documentclass[english, 10pt]{amsart}
\usepackage{amsmath}
\usepackage{graphicx}
\usepackage{amssymb}
\usepackage{amsfonts, amsthm}
\usepackage[colorlinks=true,linkcolor=blue,citecolor=blue,urlcolor=blue]{hyperref}  

\newcommand{\al}{\alpha}
\newcommand{\be}{\beta}
\newcommand{\ga}{\gamma}
\newcommand{\de}{\delta}

\newcommand{\auskommentieren}[1]{}

\newcommand{\beq}{\begin{equation}}
\newcommand{\eeq}{\end{equation}}
\newcommand{\bea}{\begin{equation}\begin{aligned}}
\newcommand{\eea}{\end{aligned}\end{equation}}

\newtheorem{theorem}{Theorem}[section]

\newtheorem{lem}[theorem]{Lemma}
\newtheorem{cor}[theorem]{Corollary}

\theoremstyle{definition}

\newtheorem{rem}[theorem]{Remark}

\newtheorem{ass}[theorem]{Assumption}

\usepackage{array} 
\usepackage{amssymb}
\usepackage{latexsym}
\usepackage{epic}
\usepackage{curves}
\usepackage{fancyhdr}
\usepackage{xspace}

\numberwithin{equation}{section}

\DeclareMathOperator{\pr}{pr}

\DeclareMathOperator{\graph}{graph}
\DeclareMathOperator{\dist}{dist}

\DeclareMathOperator{\diam}{diam}

\DeclareMathOperator{\dive}{div}
\DeclareMathOperator{\inte}{int}
\DeclareMathOperator{\width}{width}
\DeclareMathOperator{\circl}{circle}

\DeclareMathOperator{\co}{conv}

\title[Flowing the leaves of a foliation]
{Flowing the leaves of a foliation with  normal speed given by the logarithm of general curvature functions}
\author{Heiko Kr\"oner}
\address{Universit\"at Duisburg-Essen, Fakult\"at f\"ur Mathematik, Thea-Leymann-Stra\ss e 9, 45127 Essen, Germany}
\email{heiko.kroener@uni-due.de}

\begin{document}
\maketitle
\begin{abstract}
Generalizing results of Chou and Wang \cite{1} we study the flows of the leaves $(M_{\Theta})_{\Theta>0}$ of a foliation of 
$\mathbb{R}^{n+1}\setminus \{0\}$ consisting of uniformly convex hypersurfaces
in the direction of their outer normals with speeds 
$-\log(F/f)$. For quite general functions $F$ of the principal curvatures of the flow hypersurfaces  and
$f$ a smooth and positive function on $S^n$ (considered as a function of the normal)
we show that there is a distinct leaf $M_{\Theta_{*}}$ in this foliation with the property that
the flow starting from $M_{\Theta_{*}}$ converges to a translating solution of the flow
equation. Furthermore, when starting the 
flow from a leave inside $M_{\Theta_{*}}$ it shrinks to a point and when starting the 
flow from a leave outside $M_{\Theta_{*}}$ it expands to infinity. While \cite{1} considered this mechanism with $F$ equal to the Gauss curvature we allow $F$ to be among others
the elementary symmetric polynomials $H_k$. We, furthermore, show that such kind of behavior is robust 
with respect to relaxing certain assumptions at least in the rotationally symmetric and homogeneous degree one curvature function case.
 \end{abstract}

\tableofcontents

\section{Introduction}
Chou and Wang \cite{1} study a logarithmic Gauss curvature flow
of the leaves of a foliation of $\mathbb{R}^{n+1}\setminus \{0\}$ consisting
of a homothetic family of uniformly convex hypersurfaces. While their main purpose in doing so is to provide
a variational reproof of the Minkowski problem we focus in our paper
on the tool of the flow of the leaves of a foliation itself in a more general context, i.e. for other flow speeds (and depending on the flow speeds under less or more restrictive assumptions on the 
foliation).
In our paper the flow speeds in the direction of the outer unit normal of the flow hypersurfaces are of type $-\log (F/f)$
where $F$ is a curvature function of the principal curvatures and $f$ is a smooth positive function on $S^n$ which we consider via the Gauss map
also being defined on uniformly convex hypersurfaces.

Throughout the paper we make the following assumption.

\begin{ass} \label{ass1}
Either 

(i) the inverse $\tilde F$ of $F\in C^{\infty}(\Gamma_+)\cap C^0(\bar \Gamma_+)$ (where $\Gamma_+$ denotes the positive cone in $\mathbb{R}^n$ and see for a definition of the inverse of a curvature function \eqref{definition_inverse}) satisfies Assumption \ref{defi_F} 
and $f$ is a positive function on $S^n=\{x \in \mathbb{R}^{n+1}:|x|=1\}$
or  

(ii) the inverse $\tilde F$ of $F\in C^{\infty}(\bar \Gamma_+)$ satisfies Assumption \ref{defi_F_2}, 
\beq
F, \frac{\partial F}{\partial \kappa_i}>0
\eeq  
in
\beq
\inte (\bar \Gamma_+ \cap \{\kappa_1=0\}) \cup \Gamma_+
\eeq  
(Here, the interior is relative to $\{\kappa_1=0\}$ and note that $F$ is symmetric.) and the $M_{\Theta}$ are rotationally symmetric
with respect to a fixed axis. 
Furthermore, we choose this fixed axis as the $x$-axis in an appropriately chosen Euclidean coordinate system $(x, x_2,  ..., x_{n+1})$
in $\mathbb{R}^{n+1}$ and assume that $f$ is a positive function on $S^n$ which 
depends only on the $x$-coordinate and which satisfies
\begin{equation} \label{ass_for_f}
 f < c(n)f(0, ..., 0, 1)
\end{equation}
with some dimensional constant $c(n)>1$ where  $c(n)=\frac{n}{n-1}$ is a possible choice.
\end{ass}

Let us recall how the mechanism in \cite{1} works. 
Let  $(M_{\Theta})_{\Theta>0}$, $M_{\Theta}=\Theta M_0$, 
be a family of homothetic transformations of an embedded, closed, uniformly convex 
hypersurface $M_0$ in $\mathbb{R}^{n+1}$. 
There is exactly one $\Theta_{*}>0$ for which the flow starting from $M_{\Theta_{*}}$ at time $t=0$ with (outer) normal speed 
\beq \label{Gauss_speed}
-\log(K/f),
\eeq
 $K$
the Gauss curvature, 
converges to a translating solution of the flow equation. 
Here, $f$ is a positive function of the normal and $K$ the Gauss curvature of the flow hypersurfaces. The
flows with (outer) normal speed also according to \eqref{Gauss_speed} and
starting from $M_{\Theta}$ shrink to a point in the case $\Theta<\Theta_{*}$ and converge to expanding spheres in the case $\Theta>\Theta_{*}$.
The limit speed $\xi\in \mathbb{R}^{n+1}$ for the translating limit hypersurface of the flow
starting from $M_{\Theta_{*}}$ is obtained
from the necessary condition for the Minkowski problem
\begin{equation} \label{necessary}
 \int_{S^n}\frac{x_i}{e^{\xi\cdot x}f(x)}=0, \quad i=1, ..., n+1,
\end{equation}
and hence convergence to a translating hypersurface with Gauss curvature $e^{\xi\cdot x}f(x)$ is deduced. 
A necessary condition like (\ref{necessary}) 
is not known to us for the problem of finding hypersurfaces with general 
prescribed curvatures. See \cite{GuanGuan} where this difficulty is discussed e.g. for the mean curvature. 
Hence the natural goal without any further symmetry assumptions in the case $\Theta=\Theta_{*}$ for our generalization of the above mechanism 
is to obtain
convergence to a translating solution, cf. Theorem \ref{new_main_result}
for our precise main result.

Note that working with a foliation and not only with a single initial hypersurface is a kind of  natural since even on small $C^2$ perturbations of the suitably scaled sphere the flow speed has no sign in general.
The latter phenomenon is in contrast to the cases of contracting flows with mean curvature flow as the prototype \cite{Huisken0, Gage}, compare also with Gauss curvature flow \cite{Gauss_flow, AndrewsInventiones, Brendle} and expanding flows with inverse mean curvature flow as the prototype \cite{G1, Urbas}. 
While these references describe the phenomenon in the smooth parametric case of driving (especially) uniformly convex 
initial hypersurfaces after appropriate rescaling to unit spheres in the smooth topology,
 mean curvature flow appeared in \cite{Brakke} in a  weak formulation,
Gauss curvature flow in a physical application \cite{Firey} and inverse mean curvature flow \cite{Geroch} in relation with a quantity from general relativity which is monotone under the flow. 

As link of our paper to the literature we see mainly the aspect of translating solutions to geometric flows as well as
the wish to study fully nonlinear versions of e.g. the mean curvature flow. Both will be described shortly in the following.
Translating solutions appear e.g. as limiting behavior of rescaled mean curvature flow
of surfaces in the presence of type II singularities, 
see \cite{Angenent, HuiskenSinestrari}. Furthermore, translating solutions appear 
in the works \cite{Huisken, 100} as limiting behavior of 
solutions of non-parametric 
mean curvature evolution with Neumann boundary conditions.  
Translating solutions appear also in the 
limiting behavior of the 
second boundary value problem for certain non-parametric curvature flows \cite{4}
of strictly convex hypersurfaces. For further aspects of and literature about translating solutions see e.g. also \cite{SpruckXiao, BourniLangfordTinaglia2, Ilmanenoverview, Ilmanen} especially for recent results and an overview
of complete translating solutions of the mean curvature flow.
There has been interest to study fully nonlinear versions of the parametric mean curvature flow \cite{Huisken0}, already the special case (to which we also restrict our literature examples) in which the flow speed is a nonlinear function of the mean curvature has received attention in some papers, and we mention as literature for that latter case  e.g.  \cite{AndrewsCurve, Schulze, Schulze2, Sinestrari} for contracting flows and apart from the above already mentioned inverse mean curvature flow
\cite{Geroch, G1, Urbas}
also the work \cite{Heiko} for an expanding flow with normal speed given by a high power of the (multiplicative) inverse of the mean curvature where properties concerning the preservation of pinching known from  contracting flows  \cite{Andrews}  turn out to work as well. 
We refer to \cite{Lu} where the flow of certain leaves of a foliation by an anisotropic inverse and logarithmic harmonic mean curvature flow is studied.
Therein convergence to expanding spheres is shown provided the initial leave is outside a certain distinct leave. But the latter is not characterized as being critical in the sense that this phenomenon does not remain to be true for leaves enclosed by the distinct one. 
Our paper differs from  \cite{Lu} crucially since we study different flow speeds and our main issue is the existence of a critical leave which is not shown in the setting of that paper.

We introduce the setting of our paper more precisely and state our main results in 
Theorem \ref{new_main_result}. Let ($M_{\Theta})_{\Theta>0}$
 be a foliation of $\mathbb{R}^{n+1}\setminus \{0\}$ by
 embedded, closed, uniformly convex  hypersurfaces $M_{\Theta}$ 
 (i.e. in each point of $M_{\Theta}$ all principal curvatures are positive)
 where we assume
 that the parameter $\Theta$ can be viewed via that foliation as a smooth function in $\mathbb{R}^{n+1}\setminus\{0\}$ with non-vanishing gradient. W.l.o.g. we assume that the monotone ordering of the associated open convex bodies $C_{\Theta}$ of the $M_{\Theta}$ with respect to inclusion is increasing.
 Let $(X_{\Theta})_{\Theta>0}$
 be a family of
embeddings $X_{\Theta}:S^n\rightarrow \mathbb{R}^{n+1}$ of $M_{\Theta}$. We consider the evolution of convex hypersurfaces $M(t)$, parametrized by $X(\cdot, t)$, so that
\begin{equation}\label{1}
\frac{\partial X}{\partial t} = -\log(F/f)\nu
\end{equation}
with 
\begin{equation} \label{2}
X(p,0) = X_{\Theta}(p).
\end{equation}
Here, $\nu(p,t)$ denotes the unit outer normal of $M(t)$ at $X(p,t)$, $F$, $f$
satisfy Assumption \ref{ass1}, $F$ is evaluated at the principal curvatures $\kappa_i$
of
$M(t)$ and $f$ is considered to be defined on $M(t)$  via the Gauss map. 
Our main result is as follows, compare with \cite{1}.

\begin{theorem} \label{new_main_result}
(i) Let $(M_{\Theta})_{\Theta>0}$ be as above and let $(M_{\Theta})_{\Theta>0}$, $F$, $f$ satisfy Assumption \ref{ass1}. 
Then there exists $\Theta^{*}>0$ and $\xi \in \mathbb{R}^n$ 
so that the flow (\ref{1}), (\ref{2}) with initial hypersurface $X_{\Theta^{*}}$ 
converges to a translating solution of the flow equation which translates with speed $\xi$, i.e.
\begin{equation}
X(\cdot, t)-\xi t \rightarrow X^{*}
\end{equation} 
in $C^m(S^n)$, $m\in \mathbb{N}$,
for $t \rightarrow \infty$ where $X^{*}$ is the embedding of a smooth, uniformly convex hypersurface.
If $\Theta \in (0, \Theta^{*})$ 
then the solution of (\ref{1}), (\ref{2}) shrinks to
a point in finite time. 
If $\Theta \in (\Theta^{*}, \infty)$ then the diameters of the solutions expand to 
infinity as $t$ goes to infinity. 

(ii) Moreover, in the case (i) of Assumption \ref{ass1} 
the solutions converge to expanding spheres for $\Theta>\Theta_{*}$. 

\end{theorem}

Our paper is organized as follows. 
The rest  of the paper deals with the proof of Theorem \ref{new_main_result}. 
In the remaining part of this section 
we introduce some notations for curvature functions. 
Section \ref{elemantary} states for completeness some elementary facts used throughout the paper.
Section \ref{apriori_estimates} estimates the principal radii of curvature of the flow hypersurfaces 
from below and above as well as their inradii from below. All these bounds depend on the diameter
which itself is known to be bounded in Section \ref{apriori_estimates} only in case (i) of Assumption \ref{ass1}.
Using the estimates from Section \ref{apriori_estimates} we prove Theorem \ref{new_main_result} in the case (i) of Assumption \ref{ass1} in Section \ref{proof_of_main_result}.  
In Section \ref{additional_section} we use several spatially averaged speed estimates. 
In several appearing situations certain properties e.g. symmetry, convexity or a large extension of the surface in at least one direction are assumed. This allows us in a certain kind to translate  estimates of speed and of location  into each other.
 Thereby we prove a diameter bound which holds uniformly in time.
Combining this diameter bound with a priori estimates from the previous sections we then prove Theorem \ref{new_main_result} in the case (ii) of Assumption \ref{ass1}.

In the following we recall some facts about curvature functions from \cite{CP}.
Let $\Gamma \subset \mathbb{R} ^n$ denote a symmetric cone, $(\Omega, \xi^i)$ a coordinate chart in $\mathbb{R}^n$, $(g_{ij})$ a fixed positive definite $T^{0,2}(\Omega)$-tensor with inverse $(g^{ij})$ and $S=\text{Sym}(n)$ the subset of symmetric tensors in $T^{0,2}(\Omega)$.
Let $S_{\Gamma}$ be the set of the tensors $(h_{ij})$ in $S$ with eigenvalues with respect to $(g_{ij})$, i.e. eigenvalues of the $T^{1,1}(\Omega)$-tensor $(g^{ik}h_{kj})$, lying in $\Gamma$. In this setting we always consider a symmetric function $F$ defined in $\Gamma$ also as a function 
$F(\kappa_i)\equiv F(h_{ij}, g_{ij}) \equiv F(\frac{1}{2}(h_{ij}+h_{ji}), g_{ij})$ where the last expression is defined for general $(h_{ij}) \in T^{0,2}(\Omega)$. Using these interpretations we denote partial derivatives by
\begin{equation}
F_i= \frac{\partial F}{\partial \kappa_i}, \quad F_{ij}=\frac{\partial^2F}{\partial \kappa_i \partial \kappa_j}
\end{equation}
and
\begin{equation}
F^{ij}= \frac{\partial}{\partial h_{ij}}F(\frac{1}{2}(h_{ij}+h_{ji}), g_{ij}), \quad 
F^{ij, kl}= \frac{\partial^2}{\partial h_{ij}\partial h_{kl}}F(\frac{1}{2}(h_{ij}+h_{ji}), g_{ij}).
\end{equation}
 For a symmetric function $F$ in $\Gamma_{+}=\{\kappa \in \mathbb{R}^n: \kappa_i>0\}$ we define its inverse $\tilde F$ by
\begin{equation} \label{definition_inverse}
\tilde F(\kappa_i^{-1})= \frac{1}{F(\kappa_i)}, \quad (\kappa_i) \in \Gamma_{+}.
\end{equation}

Here and in the following we sometimes denote partial derivatives by indices separated by a comma for greater clarity of the presentation, sometimes we also omit the comma hereby.
Furthermore, we will often use summation convention. 

In the following we state two assumptions which summarize some technical properties for reference purposes.
\begin{ass} \label{defi_F} 
$\tilde F=\tilde F(\kappa)$, $\kappa=(\kappa_1, ..., \kappa_n)$, is a symmetric and positively homogeneous of degree $d_0$ function $\tilde F\in C^{\infty}(\Gamma_{+})\cap C^0(\bar \Gamma_+)$ with
 \begin{equation}
 \tilde F_{|\partial \Gamma_+}=0,
 \end{equation}
 \begin{equation}
 \tilde  F_i = \frac{\partial \tilde F}{\partial \kappa_i}>0 \quad \text{in }  \Gamma_{+}
 \end{equation}
 and
 \begin{equation} \label{log_concave}
 \tilde F^{ij,kl}\eta_{ij}\eta_{kl} \le \tilde F^{-1}\left(\tilde F^{ij}\eta_{ij}\right)^2-\tilde F^{ik}\tilde h^{jl}\eta_{ij}\eta
_{kl} \quad \forall \eta \in S
\end{equation}
where $(\tilde h^{ij})$ is the inverse of $(h_{ij})$.
\end{ass}

 $\tilde F$ satisfying Assumption \ref{defi_F} is said to be of class $(K)$, see \cite[Definition 2.2.1]{CP}.

First note that for all $k=1, ..., n$  the inverse of the elementary symmetric polynomial $H_k$ satisfies Assumption \ref{defi_F}, see \cite[Lemma 2.2.11]{CP}.

Second note  that Assumption \ref{defi_F} implies that 
\beq
\tilde F_n\kappa_n \ge ... \ge  \tilde F_1 \kappa_1
\eeq
provided $0<\kappa_n \le ...\le \kappa_1$,
see \cite[Lemma 2.2.4]{CP}. Note that $\log \tilde F$ is concave and that 
$\tilde F$ is concave if $d_0=1$, see \cite[Inequality (2.2.4) and Lemma 2.2.14]{CP}.

The following list of assumptions will be used to observe a certain  robustness (from the perspective of the elementary symmetric polynomials) of the mechanism in the rotationally symmetric  and homogeneity degree one case in a further direction.

\begin{ass} \label{defi_F_2} 
$\tilde F$ is a symmetric and positively homogeneous of degree $d_0=1$ function $\tilde F\in C^{\infty}(\Gamma_{+})\cap C^0(\bar \Gamma_+)$ with
 \begin{equation}
 \tilde F_{|\partial \Gamma_+}=0,
 \end{equation}
 \begin{equation}
 \tilde  F_i = \frac{\partial \tilde F}{\partial \kappa_i}>0 \quad \text{in }  \Gamma_{+},
 \end{equation}
 \begin{equation} \label{epsilon}
 \tilde F \text{ is concave}
\end{equation}
and for $0<\kappa_n\le ...\le \kappa_1$, $\kappa=(\kappa_1, ..., \kappa_n)$ holds that
\beq \label{new_condition}
\tilde F_i\kappa_i \ge \tilde F_j \kappa_j\left(\frac{\kappa_n}{\kappa_1}\right)^{\eta}, \quad 1\le j\le i \le n,
\eeq
with some  $0<\eta<1$ which does not depend on $\kappa$.
\end{ass}

We thank Oliver Schn\"urer for the interesting suggestion to study the flow of foliations 
by general speed functions and the helpful hint
to use a method from \cite{4} which allows to conclude convergence of the flow to a translating solution
when the a priori estimates are available.

\section{Elementary facts}\label{elemantary}
For convenience  and completeness we formulate relevant but obvious properties of the foliation $(M_{\Theta})_{\Theta>0}$ in the following two remarks and a lemma.
\begin{rem} \label{ordering}
For each $M_{\Theta}$ we denote the to $M_{\Theta}$ associated
open convex body by $C_{\Theta}$ and have w.l.o.g. (otherwise consider $1/\Theta$)
\begin{equation}
\Theta_1 < \Theta_2 \Rightarrow \overline{C_{\Theta_1}} \subset C_{\Theta_2}.
\end{equation}
Furthermore,  all $C_{\Theta}$ contain 0, otherwise
\begin{equation}
0<d:= \inf\{\Theta>0: \forall_{\tilde \Theta \ge \Theta}\ 0 \in C_{\tilde \Theta} \} < \infty
\end{equation}
where the last inequality is due to the fact that for $p\in \mathbb{R}^{n+1}\setminus\{0\}$ there is $\Theta(p)>0$
so that $p, -p\in C_{\Theta(p)}$ and hence also $0\in C_{\Theta(p)}$.
We conclude  $0 \in M_d$, a contradiction.
\end{rem}

\begin{rem} \label{10}
For all $r>0$ exist $\Theta_1, \Theta_2>0$ so that
\begin{equation}
\quad M_{\Theta_1} \subset B_r(0) \subset C_{\Theta_2}.
\end{equation}
\end{rem}
\begin{proof}
Let $r>0$. Existence of $\Theta_2$ as claimed is clear in view of
\begin{equation}
\overline{B_r(0)} \subset \bigcup_{\Theta>0}C_{\Theta}.
\end{equation}
Assume there are sequences $0<\Theta_k \rightarrow 0$, $x_k \in C_{\Theta_k}$, $x_k\notin B_r(0)$. W.l.o.g. assume $x_k \rightarrow x \in B_r(0)^c$. Let $p=\frac{x}{2}$. There is $\Theta=\Theta(p)>0$ so that $p \in M_{\Theta(p)}$. If $[0,x]$ meets $M_{\Theta(p)}$ tangentially in $p$ then $0 \notin C_{\Theta(p)}$ 
in view of the uniform convexity of $M_{\Theta(p)}$ which is a contradiction.  Hence
there is a neighborhood $U$ of $x$ so that for every $q \in U$ the segment $[0, q]$ meets $M_{\Theta(p)}$
non-tangentially.
This implies 
\begin{equation}
U \subset {(C_{\Theta(p)})}^c\subset (C_{\Theta_k})^c
\end{equation} 
for large $k$. On the other hand 
\begin{equation}
x_k \in U \cap C_{\Theta_k}
\end{equation}
for large $k$, a contradiction.
\end{proof}

\begin{lem} \label{11}
\begin{equation}
\frac{d}{d\Theta}H_{\Theta} >0.
\end{equation}
\end{lem}
\begin{proof}
Let $0<\Theta_1<\Theta_2<\infty$, $x \in S^n$.  In view of $D\Theta \neq 0$ there is $c_0=c_0(\Theta_1)>0$ so that
\begin{equation}
\dist(M_{\Theta_1}, M_{\Theta_2})\ge c_0 (\Theta_2-\Theta_1).
\end{equation}
For $x \in S^n$ let $y_x \in M_{\Theta_1}$ be so that
\begin{equation}
 H_{\Theta_1}(x) = xy_x >0
\end{equation}
and hence also 
\begin{equation}
 c_1 = \inf_{x \in S^n}x\frac{y_x}{|y_x|} >0.
\end{equation}
Let $y$ be the intersection of the ray starting in 0 through $y_x$ with $M_{\Theta_2}$ then 
\begin{equation}
x\cdot y \ge x \cdot y_x + c_0c_1(\Theta_2-\Theta_1)
\end{equation}
hence
\begin{equation}
\begin{aligned}
H_{\Theta_2}(x) \ge& x \cdot y_x + c_0c_1(\Theta_2-\Theta_1) \\
=& H_{\Theta_1}(x) + c_0c_1(\Theta_2-\Theta_1) 
\end{aligned}
\end{equation}
which implies 
\begin{equation}
{{\left(\frac{d}{d\Theta}{H_{\Theta}}(x)\right)}}_{|\Theta=\Theta_1} >0.
\end{equation}
\end{proof}

For completeness we state  the following well-known avoidance property (in a rather less general form) of the flows which we use several times with or without mentioning it each time, compare also with the version in Lemma \ref{global_maximum_principle}. 

\begin{lem} \label{maximum_principle}
Let $R, T>0$ and denote the open ball of radius $R$ in $\mathbb{R}^n$ around the origin by $B_R(0)$.
Let 
\beq
u_i:[0, T] \times B_R(0) \rightarrow \mathbb{R},
\eeq
$i=1,2$,
be smooth functions where we exclude for convenience
 borderline case graphs by assuming that 
\beq
|Du_i| \le c
\eeq
in $[0, T] \times B_R(0)$ 
for some positive constant $c$ and $i=1,2$.
We assume that for $t \in [0, T]$ both graphs
\beq
M_i(t)=\graph u_i(t, \cdot) =\{(u_i(t,x), x):x \in B_R(0)\} \subset \mathbb{R}^{n+1},
\eeq
$i=1,2$,
are  convex, have positive $F$-curvature and
evolve according to \eqref{1} with $F$ as in Assumption \ref{defi_F_2} where we assume that the outer unit normal
vector field $\nu^i=(\nu^i_1, ..., \nu^i_{n+1})\in \mathbb{R}^{n+1}$ of $M_i(t)$, $i=1,2$, is pointing upwards , i.e. $\nu^i_1>0$.
Let $0<t_0<T$, $x_0\in B_R(0)$, be such that 
\bea
u_1(t,x)\le&u_2(t,x) \quad \forall t \in [0, t_0) \quad \forall x \in B_R(0) \\
u_1(t_0, x_0) =& u_2(t_0, x_0).
\eea
Then there are $\Delta t, \Delta x>0$ such that
\beq
u_1(t,x)=u_2(t,x) \quad \forall (t,x)  \in [t_0-\Delta t,  t_0]  \times B_{\Delta x}(x_0).
\eeq
\end{lem}
\begin{proof}
See e.g. the lecture notes \cite{Schnuerer} in the case of mean curvature flow. Our more general case is an adaption of that argument.
\end{proof}

\section{A priori estimates for general speeds} \label{apriori_estimates}

We recall some facts about the support function of  a closed and convex hypersurface $M$ in $\mathbb{R}^{n+1}$ from \cite{1}
and follow the presentation therein closely, see also \cite{3} and \cite{5}.
The support function $H$ of $M$ is defined on $S^n$ by 
\begin{equation}
H(x) = \sup_{y\in M}x\cdot y
\end{equation}
where the dot denotes the inner product in $\mathbb{R}^{n+1}$. 
It is sometimes convenient to work with the 
homogeneous degree one extension of $H$ in $\mathbb{R}^{n+1}$ which we also denote by $H$.
$H$ is convex in $\mathbb{R}^{n+1}$ and we have
\begin{equation} \label{gradient_estimate}
\sup_{S^n}|\nabla H| \le \sup_{S^n}|H|
\end{equation}
since $H$ is the supremum of linear functions. If $M$ is strictly convex, i.e. for each $x$ in $S^n$ there is a unique point $p=p(x)$ on $M$ whose unit outer normal is $x$, $H$ is differentiable at $x$ and
\begin{equation}
p_{\alpha} = \frac{\partial H}{\partial x_{\alpha}}, \quad \alpha =1, ..., n+1.
\end{equation}
Furthermore, given an orthonormal frame fields $e_1, ..., e_n$ on $S^n$ and denoting covariant differentiation with respect to $e_{i}$ by $\nabla_{i}$ the eigenvalues of $(\nabla_{i} \nabla_{j}H+H\delta_{ij})_{i, j=1, ..., n}$, are the principal radii of curvature at $p(x)$. When $H$ is viewed as a homogeneous function over $\mathbb{R}^{n+1}$, the principal radii of curvature of $M$ are also equal to the non-zero eigenvalues of the Hessian 
\begin{equation}
\left(\frac{\partial^2 H}{\partial x_{\alpha} \partial x_{\beta}}\right)_{\alpha, \beta=1, ..., n+1}
\end{equation}
on $S^n$. 

We begin with a reformulation of Equation (\ref{1}) locally in Euclidean space, cf. Equation (\ref{1_6_}). 
Let $H(\cdot, t): S^n \rightarrow \mathbb{R}$ be the support function of $M(t)$ where we denote its homogeneous degree one extension to $\mathbb{R}^{n+1}$ again by $H(\cdot, t)$ and let $p(\cdot)=p(\cdot, t)$ denote the inverse of the Gauss map $M(t)\rightarrow S^n$.
Using
\begin{equation}
 \frac{H}{\partial t}(x,t)= x \cdot \frac{\partial X}{\partial t}(p(x), t), \quad x \in S^n,
\end{equation}
we rewrite
problem (\ref{1}) as the following initial value problem for $H$
\begin{equation} \label{flow_H}
 \begin{aligned}
  \frac{\partial H}{\partial t}  =& \log \frac{f}{F}=\log \tilde Ff \\
  H(x,0) =& H_{\Theta }(x)
 \end{aligned}
\end{equation}
where $H_{\Theta}$ is the support function for $M_{\Theta}$ and $\tilde F$ a function of the
principal radii $r_i=\kappa_i^{-1}$ defined by 
\begin{equation} \label{F_explizit_in_terms_of_eigenvalue}
F=F(\kappa_i)=\tilde F(\kappa_i^{-1})^{-1}=\tilde F(r_i)^{-1}.
\end{equation}
We set $u(y,t)=H(y, -1, t)$, $y \in \mathbb{R}^n$. Then $u(\cdot, t)$ is convex and the principal radii $r_i$ of $X(\cdot, t)$ in $p(x,t)$, $x \in S^n$, are given as nonzero zeros of the
 equation
 \begin{equation} \label{matrix}
 \det B =0
 \end{equation}
 in the variable $r$
 where $B=(B_{\alpha \beta})_{\alpha, \beta =0, ..., n}$ with
\begin{equation} \label{defi_B}
(B_{\alpha \beta}) = \left(
 \begin{matrix}
  -\frac{\lambda^2}{r} & y_1 & ... & y_n \\
  y_1 & \lambda u_{11}-r & ... & \lambda u_{1n} \\
  ... \\
  y_n & \lambda u_{n1} & ... & \lambda u_{nn}-r
 \end{matrix}
 \right), 
\end{equation}
$\lambda = (1+y_1^2+ ... + y_n^2)^{\frac{1}{2}}$ and $x$ and $y$ are related by 
\begin{equation}
x=(y,-1)/\sqrt{1+|y|^2},
\end{equation}
cf. \cite[page 16]{3}, and note that we have rewritten the equation therein slightly. 
Furthermore, we have
\begin{equation}
 \frac{\partial u}{\partial t}(y,t) = \sqrt{1+|y|^2}\frac{\partial H}{\partial t}(x,t).
\end{equation}
Extending $f$ to be a homogeneous function of degree 0 in $\mathbb{R}^{n+1}$ we obtain the local representation of (\ref{1}) in terms of $u$
\begin{equation} \label{1_6}
 \frac{\partial u}{\partial t} = \sqrt{1+|y|^2}\log \tilde F + l(y), \quad y\in \mathbb{R}^n,
\end{equation}
where
\begin{equation}
 l(y) = \sqrt{1+|y|^2}\log f(y, -1)
\end{equation}
and $\tilde F$ is evaluated at the zeros $r_i$ of Equation (\ref{matrix}). 
For technical reasons we rewrite this equation slightly by using the homogeneity of $\tilde F$ 
\begin{equation} \label{1_6_}
 \frac{\partial u}{\partial t} = \sqrt{1+|y|^2}\log \tilde F(\lambda^{-3}r_i) + g(y), \quad y\in \mathbb{R}^n,
\end{equation}
where
\begin{equation}
 g(y) = l(y) + 3d_0  \lambda \log \lambda.
\end{equation}

From the maximum principle one gets an analogous comparison principle as \cite[Lemma 2.1]{1} which implies uniqueness of a solution of (\ref{flow_H}). Compare this also with the more special and local version in Lemma \ref{maximum_principle}.

\begin{lem} \label{global_maximum_principle}
 For $i=1,2$ let $f_i$ be two positive $C^2$-functions on $S^n$
 and $H_i$ $C^{2,1}$-solutions of 
 \begin{equation}
  \frac{\partial H_i}{\partial t} = \log \tilde Ff_i.
 \end{equation}
If $H_1(x,0)\le H_2(x,0)$ and $f_1(x)\le f_2(x)$ on $S^n$ then $H_1 \le H_2$
 for all $t>0$ and $H_1<H_2$ unless $H_1\equiv H_2$.
\end{lem}

In the following we will always assume that $H\in C^{\infty}(S^n\times [0, T])$ is a solution
of (\ref{flow_H}). We denote the outer and inner radii of the hypersurface
$X(\cdot, t)$ determined by $H(\cdot,t)$ by  $R(t)$ and $r(t)$, respectively, and set
\begin{equation}
 R_0 = \sup\{R(t): t \in [0 ,T]\}
\end{equation}
and
\begin{equation}
 r_0 =\inf\{r(t):t\in [0, T]\}.
\end{equation}
The goal of the present section is to estimate the principal radii of curvatures of $X(\cdot, t)$ from below and above in terms of 
$r_0$, $R_0$ and initial data.

Lemma \ref{lemma_2.2}, Lemma \ref{lemma_2_3}, Lemma \ref{lemma_1}, Corollary \ref{korollar} and Lemma \ref{lemma_2} which will follow below 
are concerning their formulation the same as the corresponding ones in \cite{1} but refer here to a different flow.
We state them for the convenience of the reader and present proofs when differences to \cite{1}
appear. 
We begin with two lemmas needed in the following.
\begin{lem} \label{lemma_2.2}
 Let $r$ and $R$ be the inner and outer radii of a uniformly convex hypersurface $X$ respectively.
 Then there exists a dimensional constant $C$ such that
 \begin{equation}
  \frac{R^2}{r}\le C\sup \{R(x, \xi): x, \xi \in S^n\},
 \end{equation}
 where $R(x, \xi)$ is the principal radius of curvature of $X$ at the point with normal $x$ and 
 along the direction  $\xi$.
\end{lem}
\begin{proof}
 See \cite[Lemma 2.2]{1}.
\end{proof}
\begin{lem} \label{lemma_2_3}
Let $a(t), b(t) \in C^1([0, T])$ and $a(t)<b(t)$ for all $t$. Then there exists 
 $h(t)\in C^{0,1}([0, T])$ such that
 
 i) $a(t)-2M \le h(t)\le b(t)+2M$,
 
 ii) $\sup \{ \frac{|h(t_1)-h(t_2)|} {|t_1-t_2|} : t_1, t_2 \in [0, T] \} \le 2 \max \{\sup_t b'(t)
 , \sup_t(-a'(t))\}$,\\
 where $M=\sup_t(b(t)-a(t))$.
\end{lem}
\begin{proof}
 See \cite[Lemma 2.3]{1}
\end{proof}

In the following lemma we prove an upper bound for the principal radii of curvature.

\begin{lem} \label{lemma_1}
We assume part (i) of Assumption \ref{ass1}.
 For any $\gamma \in (1,2]$ there exists a constant $c_{\gamma}$ which may depend on initial data such that
 \begin{equation}
 \begin{aligned}
  \sup  \{H_{\xi\xi}(x,t): (x,t) \in S^n \times  [0, T], \xi \in T_xS^n, |\xi|=1\} 
  \le c_{\gamma}(1+D^{\gamma}),
  \end{aligned}
 \end{equation}
 where $D=\sup\{d(t): t\in [0, T]\}$ and $d(t)$ is the diameter of $X(\cdot, t)$. 
\end{lem}
\begin{proof} 
We adapt the proof of \cite[Lemma 2.4]{1} by including the case of a more general speed function.
Therefore we develop a method 
based on the novelty that we identify principal radii as eigenvalues of certain lower dimensional matrices building on the representation \eqref{matrix}. 
This allows us in a very convenient way to closely follow, in the case of the general speed function (instead of the Gauss curvature), the argumentation of the proof of \cite[Lemma 2.4]{1}. In this proof,  a convenient and special parametrization is used which can be traced back to older works. In addition to the formal analogy, the mathematical analogy holds true since we use Assumption \ref{ass1} (i). We will revisit this framework in the proof of Lemma \ref{new_lemma_1} where we have only part (ii) instead of part  (i) of Assumption \ref{ass1} available yielding a crucially weaker conclusion.

In the following we precisely repeat the formulas from \cite[Lemma 2.4]{1} for convenience.
 Applying Lemma \ref{lemma_2_3} to the functions $-H(-e_i, t)$ and $H(e_i, t)$
 where $\pm e_i$ are the intersection points of $S^n$ 
 with the $x_i$-axis, $i=1, ..., n+1$, we obtain
 $p_i(t)$ so that
 \begin{equation}
  -H(-e_i, t)-2D \le p_i(t)\le H(e_i, t)+2D
 \end{equation}
 and
 \begin{equation}
  \begin{aligned}
  \sup\left\{\frac{|p_i(t_1)-p_i(t_2)|}{|t_1-t_2|}:t_1, t_2 \in [0, T]\right\} \\
 \le 2 \sup\{H_t(x,t):(x,t)\in S^n\times [0,T]\}.
 \end{aligned} 
 \end{equation}
We have
 \begin{equation} \label{2.2}
  \left|H(x,t)-\sum_{i=1}^{n+1}p_i(t)x_i\right| \le cD \quad \text{for } (x,t)\in S^n\times [0,T],
 \end{equation}
 and by (1.1)
 \begin{equation} \label{2.3}
 \sum_{i=1}^{n+1}|H_i(x,t)-p_i|^2 \le cD^2.
 \end{equation}
 Let 
 \begin{equation}
  \Phi(x,t) = H_{\xi\xi}(x,t)+\left[1+\sum_{i=1}^{n+1}|H_i(x,t)-p_i(t)|^2\right]^{\frac{\gamma}{2}}
 \end{equation}
 where $\gamma \in (1,2]$. Suppose that the supremum 
 \begin{equation}
  \sup\{\Phi(x,t):(x,t)\in S^n\times [0, T], \xi \text{ tangential to } S^n, |\xi|=1 \}
 \end{equation}
 is attained at the south pole $x= (0, ..., 0, -1)$ at $t=\bar t>0$ and in the direction $\xi=e_i$.
 For any $x$ on the south hemisphere, let
 \begin{equation}
  \xi(x) = \left(\sqrt{1-x_1^2}, -\frac{x_1x_2}{\sqrt{1-x_1^2}}, ..., -\frac{x_1x_{n+1}}{1-x_1^2}\right).
 \end{equation}

We perform the calculations in an Euclidean setting which can be achieved by considering  the restriction  $u$ of $H$ on $x_{n+1}=-1$. 
Due to the homogeneity of $H$ we obtain
\begin{equation}
 \begin{aligned}
  \sum_{i=1}^{n+1}(H_i&-p_i)^2(x,t) \\
  & = \sum_{i=1}^n(u_i(y,t)-p_i(t))^2+\left|u(y,t)+p_{n+1}-\sum_{i=1}^ny_iu_i(y,t)\right|^2
 \end{aligned}
\end{equation}
and
\begin{equation}
 H_{\xi\xi}(x,t) = u_{11}(y,t)\frac{(1+y_1^2+ ... + y_n^2)^{\frac{3}{2}}}{1+y_2^2+ ... + y_n^2},
\end{equation}
where $y=-(x_1, ..., x_n)/x_{n+1}$ in $\mathbb{R}^n$. The function 
\begin{equation}
\begin{aligned}
\varphi(y,t)=&u_{11}\frac{(1+y_1^2+ ...+y_n^2)^{\frac{3}{2}}}{1+y_2^2+ ...+ y_n^2}\\
& +\left[1+\sum (u_i-p_i)^2+|u+p_{n+1}-\sum y_i u_i|^2\right]^{\frac{\gamma}{2}}
\end{aligned}
\end{equation}
attains its maximum at $(y,t)=(0, \bar t)$ where we may w.l.o.g. assume that
the Hessian of $u$ at $(0, \bar t)$ is diagonal. Hence at $(0, \bar t)$
we have for each $k$,
\begin{equation} \label{phi_inequality}
 \begin{aligned}
  0 \le \varphi_t =& u_{11t}+\gamma [(u_i-p_i)(u_{it}-p_{i,t})+(u+p_{n+1})(u_t+p_{n+1,t})]
  Q^{\frac{\gamma-2}{2}}, \\
  0 =& \varphi_k = u_{11k}+\gamma(u_i-p_i)u_{ik}Q^{\frac{\gamma-2}{2}}
 \end{aligned}
\end{equation}
and
\begin{equation} \label{phi_inequality2}
 \begin{aligned}
  0 \ge& \varphi_{kk} = u_{kk11}+\tau_k u_{11}+\gamma[u_{kk}^2+(u_i-p_i)u_{ikk}-(u+p_{n+1})u_{kk}]
  Q^{\frac{\gamma-2}{2}}\\
  & +\gamma(\gamma-2)(u_i-p_i)^2u_{ik}^2Q^{\frac{\gamma-4}{2}},
 \end{aligned}
\end{equation}
where $Q=1+\sum (u_i-p_i)^2+(u+p_{n+1})^2$, $\tau_k=1$ if $k>1$, $\tau_1=3$ and 
$p_{i,t}=\frac{dp_i}{dt}$. 

Here we deviate formally from following \cite[Lemma 2.4]{1}.
Instead we first explain how to express the principal radii of curvature as eigenvalues of matrices in $Sym(n)$ in certain cases which will in the sequel be useful for the expression of derivatives of $u$. 
Note that the formal motivation is that it is well-known that any derivatives of $\tilde F=\tilde F(r_i)$, cf. \eqref{F_explizit_in_terms_of_eigenvalue}, with respect to the variables $y_i$, $i=1, ..., n$,
can by chain rule be very conveniently expressed in terms of the derivative of the curvature function $\tilde F$ (considered as being defined in $Sym (n)$ via the vector of eigenvalues) 
and derivatives with respect to the $y_i$ of the entries of a symmetric matrix valued function -- if there is such -- with eigenvalues given by $r_i$. 
In a straightforward way 
such a matrix valued function and thereby such kind of decomposition does not seem to be available here. But since we know that the Hessian of $u$ in $(0, \bar t)$
is diagonal we only need  first order and pure second order (not the mixed ones) partial derivatives of $u$ in the following expressions.
Since the first order and the pure second order partial derivative is  the first order and the second order derivative, respectively,  of the same function when being restricted 
to a coordinate function (while fixing all the other coordinates), it suffices to realize the function \eqref{F_explizit_in_terms_of_eigenvalue} for each partial derivative by a different decomposition. These decompositions consist in each case of an individually chosen matrix valued inner function and the same outer function given by $\tilde F$ 
(considered as being defined in $Sym (n)$ via the vector of eigenvalues).

In order to construct the matrix valued inner function which we will use to express partial derivatives with respect to $y_1$ we proceed as follows. Since the Hessian $(u_{ij})$ is diagonal in $(0, \bar t)=(y, \bar t)$ 
also $B$
is diagonal. 
We fix $y_i=0$, $i=2, ..., n$, and vary $y_1$ for a moment. In this case we rewrite Equation (\ref{matrix})  by using the matrices $B^1=(B_{ij})_{i,j =1, ..., n}$ and $B^2=(B_{ij})_{i,j=2, ..., n}$  which are sub-matrices of $(B_{\alpha \beta})_{\alpha, \beta=0, ..., n}$ as follows. We have for $r\neq 0$ that
\begin{equation} \label{spezGl}
\begin{aligned}
& \det B =0 \\
\Leftrightarrow & \det \left(
 \begin{matrix}
  -\frac{\lambda^2}{r} & y_1 & 0 & ... & 0 \\
  y_1 & \lambda u_{11}-r & \lambda u_{12} & ... & \lambda u_{1n} \\
  0 & \lambda u_{21} & \lambda u_{22}-r& ...& \lambda u_{2n} \\
  ... \\
  0 & \lambda u_{n1} & \lambda u_{n2} &... &  \lambda u_{nn}-r
 \end{matrix}
 \right) =0 \\
 \Leftrightarrow & -\frac{\lambda^2}{r}\det B^1 - y_1^2 \det B^2 =0 \\
  \Leftrightarrow & \det B^1 + \frac{y_1^2 r}{\lambda^2}\det B^2 =0 \\
  \Leftrightarrow & \det \left(
 \begin{matrix}
   \lambda u_{11}-r\left(1-\frac{y_1^2}{\lambda^2}\right) & \lambda u_{12} &  ... & \lambda u_{1n} \\
   \lambda u_{21} & \lambda u_{22}-r & ...& \lambda u_{2n} \\
  ... \\
   \lambda u_{n1} & \lambda u_{n2}& ... & \lambda u_{nn}-r
 \end{matrix}
 \right) =0 \\
  \Leftrightarrow & \det \left(
 \begin{matrix}
   \lambda^3 u_{11}-r& \lambda^2 u_{12}& ... & \lambda^2 u_{1n} \\
   \lambda^2u_{21} & \lambda u_{22}-r & ...& \lambda u_{2n} \\
  ... \\
   \lambda^2 u_{n1} & \lambda u_{n2}&... &  \lambda u_{nn}-r
 \end{matrix}
 \right) =0. 
 \end{aligned}
\end{equation}
From the last line we realize that a solution $r$ can be interpreted as eigenvalue of an artificially defined matrix 
\beq
(a^1_{ij}) = 
  \left(
\begin{matrix}
\lambda^3 u_{11}& \lambda^2 u_{12} & ... & \lambda^2 u_{1n} \\
\lambda^2 u_{21} & \lambda u_{22} & ... & \lambda u_{2n} \\
  ... \\
   \lambda^2 u_{n1} & \lambda u_{n2} & ...& \lambda u_{nn}
 \end{matrix}
 \right)
\eeq
which arises as a $\lambda$ multiple of the Hessian of $u$ at which the first row is multiplied by $\lambda$ and after this also  the first column
of the resulting matrix. We further abbreviate
\begin{equation}
  (a_{ij}) = 
  \left(
\begin{matrix}
\lambda u_{11}& ... & \lambda u_{1n} \\
  ... \\
   \lambda u_{n1} & ... & \lambda u_{nn}
 \end{matrix},
 \right).
 \quad
 \end{equation}
We summarize that the zeros of Equation (\ref{matrix}) can be written as eigenvalues of the matrix $(a^1_{ij})$ as long as $y_2=...=y_n=0$ and $y_1$ is variable. 

An analogous calculation shows that this phenomenon is not specific for the first variable. Defining for $r=1, ..., n$  the matrix $(a^r_{ij})$ as the matrix which is obtained by multiplying row $r$ and then column $r$ in $(a_{ij})$ with $\lambda$ we can write the zeros of Equation (\ref{matrix}) as eigenvalues of the matrix $(a^r_{ij})$  in the case where we vary $y_r$, $r$ fixed, and fix $y_i=0$ for $i \neq r$.

Applying these deliberations for the purpose of expressing pure derivatives up to the second order of $\tilde F= \tilde F(\lambda^{-3}r_i)$ from (\ref{1_6_})  
we may rewrite
\begin{equation}
\tilde F = \tilde F(\lambda^{-3}r_i) = \tilde F(\tilde a^r_{ij})
\end{equation}
where 
\begin{equation}
(\tilde a^r_{ij})= \lambda^{-3}(a^r_{ij})
\end{equation}
if $(y, t)=(0, ..., 0, y_r, 0, ..., 0, t)$. Note that deviating from our explanation above we indeed consider also the extra factor $\lambda^{-3}$ here but there holds 
\beq
\lambda=\lambda_{kk}=1, \quad \lambda_k=0, \quad k=1, ..., n,
\eeq 
in $y=0$ which will make no problems.
We have in $(0,\bar t)$ that
\begin{equation} \label{in_point_case}
\frac{\partial \tilde F}{\partial y_k} = \tilde F^{ii}a^k_{ii,k} \quad \wedge \quad \frac{\partial^2\tilde F}{\partial {y^2_k}}= -3 \tilde F^{ii}a^k_{ii}+\tilde F^{ii}a^k_{ii,kk}+\tilde F^{ij,rs}a^k_{ij,k}a^k_{rs,k}
\end{equation}
where we do not sum over $k$ and where we used \cite[Lemma 2.1.9]{CP} to deduce that $\tilde F^{ij}$ is diagonal. (Still we will in the sequel formally sometimes sum over the full range of $i,j$).
 Note that $\tilde F^{ij}$ and $\tilde F^{ij, rs}$ are here evaluated in $\tilde a^k_{ij}=a^k_{ij}$ in the point $(0, \bar t)$. But the latter is the same for all $k$ leading to the simplified expression \eqref{in_point_case}.

Now we can recapitulate the mathematical mechanism from \cite[Lemma 2.4]{1}
in a formally expanded shape.
Differentiating (\ref{1_6_}) gives in $(0, \bar t)$ that
\begin{equation} \label{deriv_u}
 \begin{aligned}
  u_{kt} =& \left(1+|y|^2\right)^{-\frac{1}{2}}y_k \log \tilde F + \sqrt{1+|y|^2}\frac{1}{\tilde F}
  \tilde F^{ij}\tilde a^k_{ij,k}+g_k \\
  u_{kkt}=& \log \tilde F - \frac{1}{\tilde F^2}\tilde F^{ij}\tilde a^k_{ij,k}\tilde F^{rs}\tilde a^k_{rs,k}
  +\frac{1}{\tilde F}\tilde F^{ij,rs}\tilde a^k_{ij,k}\tilde a^k_{rs,k}\\
  & +\frac{1}{\tilde F}\tilde F^{ij} a^k_{ij,kk}-\frac{3}{\tilde F}\tilde F^{ii}a^k_{ii}+g_{kk},
  \end{aligned}
\end{equation}
here, we do not sum over $k$.
Coming back to \eqref{phi_inequality} and  \eqref{phi_inequality2}  and summarizing and rewriting the resulting inequalities we arrive
in $(0, \bar t)$ at
\begin{equation} \label{term_1000}
\begin{aligned}
0 \ge& \sum_{k,l} \frac{1}{\tilde F}\tilde F^{kl}\varphi_{kl}-\varphi_t \\
=& \sum_k \frac{1}{\tilde F}\tilde F^{kk}\varphi_{kk}-\varphi_t \\
=& \sum_k\frac{1}{\tilde F} \tilde F^{kk}u_{kk11}+\frac{1}{\tilde F}\tilde F^{kk}u_{11}\tau_k \\
& + \gamma \{\sum_k\frac{1}{\tilde F}\tilde F^{kk}u_{kk}^2 [1+\frac{(\gamma-2)(u_k-p_k)^2}
{1+\sum (u_i-p_i)^2
+(u+p_{n+1})^2}]\\
& +\sum_{i}(u_i-p_i)(\frac{1}{\tilde F}\sum_{r,s}\tilde F^{rs}u_{irs}-u_{it})-\frac{1}{\tilde F}\sum_k
\tilde F^{kk}u_{kk}(u+p_{n+1})\\
& -(u+p_{n+1})(u_t+p_{n+1,t})+\sum_i(u_i-p_i)p_{i,t}\}Q^{\frac{\gamma-2}{2}}-u_{11t}
\end{aligned}
\end{equation}
and we can estimate this further from below (where we drop the summation signs) by
\begin{equation} \label{3.24}
\begin{aligned}
&\frac{1}{\tilde F}\tilde F^{kk}u_{11}-\log \tilde F+\frac{1}{\tilde F^2}
\tilde F^{ij}\tilde a^1_{ij,1}\tilde F^{rs}\tilde a^1_{rs,1}-\frac{1}{\tilde F}\tilde F^{ij,rs}\tilde a^1_{ij,1}\tilde a^1_{rs,1} \\
& -\frac{1}{\tilde F}\tilde F^{kk} a^1_{kk,11}+3d_0-g_{11}  + \frac{1}{\tilde F}\tilde F^{kk}u_{kk11} \\
& + \gamma \{(\gamma-1)\frac{1}{\tilde F}\tilde F^{kk}u_{kk}^2-(u_i-p_i)g_i\\
& +\frac{1}{\tilde F} \tilde F^{rs}(u_{rsi}-\tilde a^i_{rs,i})(u_i-p_i) -\frac{1}{\tilde F}\tilde F^{kk}u_{kk}(u+p_{n+1}) \\
& -(u+p_{n+1})(u_t+p_{n+1, t})+(u_i-p_i)p_{i,t}\}
Q^{\frac{\gamma-2}{2}} 
\end{aligned}
\end{equation}
where we used \eqref{deriv_u}. Note that the factor $\gamma-1$ comes into play by 
replacing the factor 1 by $(\gamma-1)$ in certain positive summands (after having rewritten the $[...]$ bracket as a single fraction) in the third line from below in \eqref{term_1000} and thereby obtaining an estimate from below. We continue our estimate by estimating \eqref{3.24} from below by the term
\begin{equation} \label{3.43}
\begin{aligned}
&\frac{1}{\tilde F}\tilde F^{kk}u_{11}-\log \tilde F \\
& +\frac{1}{\tilde F}\tilde F^{kk}(u_{kk11}- a^1_{kk,11})+3d_0-g_{11} \\
& + \gamma \{(\gamma-1)\frac{1}{\tilde F}\tilde F^{kk}u_{kk}^2-(u_i-p_i)g_i+\frac{1}{\tilde F}\tilde F^{rs}(u_{rsi}-\tilde a^i_{rs,i})(u_i-p_i)\\
& - d_0(u+p_{n+1})-(u+p_{n+1})(u_t+p_{n+1,t})+(u_i-p_i)p_{i,t}\}Q^{\frac{\gamma-2}{2}}
\end{aligned}
\end{equation}
where we used 
 (\ref{log_concave}).
We have in $(0, \bar t)$ that
\begin{equation}
u_{ij}=\tilde a^r_{ij}=a^r_{ij} \quad \wedge\quad u_{ij,k} = \tilde a^r_{ij,k} =a^r_{ij,k} \quad \wedge \quad
 a^1_{rr,11} = \mu u_{rr}+u_{rr;11} 
\end{equation}
where $\mu=3$ if $r=1$ and $\mu=1$ if $r\neq 1$ in the last equation.
We conclude that in $(0, \bar t)$ 
\begin{equation} \label{analoge_Stelle_1}
\begin{aligned}
0 \ge & \frac{1}{\tilde F}\tilde F^{kk}u_{11}-\log \tilde F  -g_{11}
\\
& + \gamma \{(\gamma-1)\frac{1}{\tilde F}\tilde F^{kk}u_{kk}^2-(u_i-p_i)g_i-d_0 (u+p_{n+1})\\
& -(u+p_{n+1})(u_t+p_{n+1,t})+(u_i-p_i)p_{i,t}\}Q^{\frac{\gamma-2}{2}}.
\end{aligned}
\end{equation}
W.l.o.g. let us assume $u_{11}\ge ... \ge u_{nn}$ and we define the ratio $\mu =\frac{u_{11}}{u_{nn}}$. In view
of \cite[Lemma 2.2.4]{CP} and the homogeneity of $\tilde F$
we have
\begin{equation} \label{extra_1}
\frac{1}{\tilde F}\sum_k\tilde F^{kk}u_{11}=\sum_k \frac{1}{\tilde F}\tilde F^{kk}u_{11}
\ge  \mu \frac{1}{\tilde F}\tilde F^{nn}u_{nn}\ge d_0 \frac{\mu}{n}
\end{equation}
where we do not sum over $n$ 
and
\begin{equation}\label{extra_2}
 \sum_k\frac{1}{\tilde F}\tilde F^{kk}u_{kk}^2 \ge  \frac{1}{\tilde F}\sum_k\tilde F^{kk}u_{kk}u_{nn} =
 d_0 \frac{u_{11}}{\mu}.
\end{equation}
From (\ref{2.2}) and (\ref{2.3}) we deduce  that $|u+p_{n+1}|\le cD$ and $|u_i-p_i|\le cD$. Summarized we obtain 
\begin{equation}
\mu + (c D)^{\gamma-2} \frac{u_{11}}{\mu}
\le \log \tilde F +c + cQ^{\frac{\gamma-2}{2}}D(1+|u_t|+|H_t|)
\end{equation}
and hence
\begin{equation}
(c D)^{2-\gamma} \mu +\frac{u_{11}}{\mu} \le cD^{2-\gamma}\log u_{11} + cD^{2-\gamma}+cD(1+\log u_{11}).
\end{equation}
Now  estimating
\beq
u_{11}=\mu \frac{u_{11}}{\mu}
\eeq
by the product of the right-hand sides of the previous inequalities implies the claim.

\end{proof}

\begin{cor} \label{korollar}
We assume part (i) of Assumption \ref{ass1}.
For any $\gamma \in (1,2]$ there exists $\delta=\delta(\gamma)>0$ such that
\begin{equation}
r(t) \ge \frac{\delta R(t)^2}{1+\sup_{\tau\le t}R^{\gamma}(\tau)}.
\end{equation}
\end{cor}
\begin{proof}
Use Lemma \ref{lemma_2.2} and Lemma \ref{lemma_1}.
\end{proof}

In the following lemma we estimate $H_t$ from below. In view of Lemma \ref{lemma_1} and Equation (\ref{flow_H}) this immediately implies a lower bound for the principal radii of curvature.

\begin{lem} \label{lemma_2}
We assume part (i) of Assumption \ref{ass1}.
There exists a constant $c$ depending only on $n$, $r_0$, $R_0$, $f$ and initial data such that
\begin{equation}
\inf \{H_t(x,t):(x,t)\in S^n \times [0, T]\} \ge -c.
\end{equation}
\end{lem}
\begin{proof} We adapt the proof of \cite[Lemma 2.6]{1} where the first part is word by word the same as in that reference and  is presented here in order to make clear the setting and the second part 
requires some adaptions.
Let 
\begin{equation}
q(t) = \frac{1}{|S^n|}\int_{S^n}x H(x,t)d \sigma(x)
\end{equation}
be the Steiner point of $X(\cdot, t)$. Then there exists a positive $\delta$ which depends only on $n$, $r_0$
and $R_0$ so that 
\begin{equation}
H(x,t)-q(t)\cdot x \ge 2\delta.
\end{equation}
We assume that the function 
\begin{equation}
\psi(x,t) = \frac{H_t(x,t)}{H(x,t)-x \cdot q(t)-\delta}
\end{equation}
attains its negative infimum on $S^n \times[0, T]$ at $x=(0, ..., 0, -1)$ and
$\bar t \in (0, T]$ and that $(u_{ij})$ is diagonal. Let $u$ be the restriction of $H$ to $x_{n+1}=-1$ as before. Then
\begin{equation}
\psi(y,t) = \frac{u_t(y,t)}{u(y,t)-q(t)\cdot (y, -1)-\delta \sqrt{1+|y|^2}}
\end{equation}
attains its negative minimum at $(0, \bar t)$. Hence in this point we have
\begin{equation} \label{psi_time_derivative}
0 \ge \psi_t = \frac{u_{tt}}{u+q_{n+1}(t)-\delta}-\frac{u_t(u_t+\frac{dq_{n+1}}{dt})}{(u+q_{n+1}(t)-\delta)^2},
\end{equation}
\begin{equation} \label{ableitung0}
0 = \psi_k = \frac{u_{tk}}{u+q_{n+1}(t)-\delta}-\frac{u_t(u_k-q_k(t))}{(u+q_{n+1}(t)-\delta)^2}
\end{equation}
and
\begin{equation} \label{psi_space_derivative}
0 \le \psi_{kk} = \frac{u_{tkk}}{u+q_{n+1}(t)-\delta}-\frac{u_t u_{kk}}{(u+q_{n+1}(t)-\delta)^2}
+ \frac{\delta u_t}{(u+q_{n+1}(t)-\delta)^2}.
\end{equation}
Using the notation from the proof of Lemma \ref{lemma_1} we get on the other hand by differentiating (\ref{1_6_}) that in $(0, \bar t)$
\begin{equation} \label{second_time_derivative_u}
u_{tt}= \frac{1}{\tilde F}\tilde F^{ij}u_{ijt}.
\end{equation}
We have in $(0, \bar t)$ using that $(\tilde F^{ij})$ is diagonal 
\begin{equation} \label{fkk}
\begin{aligned}
0 \le& \sum \frac{1}{\tilde F}\tilde F^{kk}\psi_{kk}-\psi_t \\
\le& \frac{\delta u_t \frac{1}{\tilde F}\sum \tilde F^{kk}-\frac{1}{\tilde F}\tilde F^{kk}u_{kk}u_t+u_t(u_t+\frac{dq_{n+1}}{dt})}{(u+q_{n+1}-\delta)^2}
\end{aligned}
\end{equation}
where we used \eqref{psi_time_derivative}, \eqref{psi_space_derivative} and \eqref{second_time_derivative_u} (all partly) repeatedly.
Since $u_t$  is negative at $(0, \bar t)$, it follows that
\begin{equation} \label{growth}
\begin{aligned}
\frac{1}{\tilde F} \sum_k \tilde F^{kk} \le& \frac{c}{\delta}(1+|u_t|) \\
\le& \frac{c}{\delta}(1+ \log \tilde F^{-1})
\end{aligned}
\end{equation}
where we used the homogeneity of $\tilde F$ and where $c=c(f, R_0)$. Furthermore, we applied the logarithm law saying that $\log\left(\frac{1}{r}\right)=-\log r$ for
$r>0$.

We choose $i_0 \in \{1, ..., n\}$ such that
\begin{equation}
u_{i_0i_0} = \min_{1 \le i \le n}u_{ii}>0
\end{equation}
and hence 
\begin{equation}
u_{i_0i_0}^{d_0}\tilde F(1, ..., 1) \le \tilde F=\frac{1}{d_0}\tilde F^{ii}u_{ii} \le c \tilde F^{i_0i_0}u_{i_0i_0}
\end{equation}
in view of the homogeneity of $\tilde F$ and \cite[Lemma 2.2.4]{CP}.
Hence we estimate
\begin{equation}
\begin{aligned}
\sum_k \tilde F^{kk} &\ge \tilde F^{i_0i_0} \\
& \ge \frac{\tilde F}{cu_{i_0i_0}}
\end{aligned}
\end{equation}
and deduce from (\ref{growth}) that
\begin{equation}
(u_{i_0i_0})^{-1}\le c(1+\log ((u_{i_0i_0})^{-1})
\end{equation}
so that $u_{i_0i_0} \ge c>0$ and the claim follows.

\end{proof}

We need versions of Lemma \ref{lemma_1}, Corollary \ref{korollar} and Lemma \ref{lemma_2} which
hold in the case (ii) of Assumption \ref{ass1}.
For Lemma \ref{lemma_1} we obtain the following analogon.
\begin{lem} \label{new_lemma_1}
We assume case (ii) of Assumption \ref{ass1}. There exist constants $c_1, c_2>0$ such that
 \begin{equation}
 \begin{aligned}
  \sup  \{H_{\xi\xi}(x,t): (x,t) \in S^n \times  [0, T], \xi \in T_xS^n, |\xi|=1\} 
  \le c_1(1+D^{c_2}),
  \end{aligned}
 \end{equation}
 where $D=\sup\{d(t): t\in [0, T]\}$ and $d(t)$ is the diameter of $X(\cdot, t)$. 
\end{lem}
\begin{proof} 
We follow the proof of Lemma \ref{lemma_1} word by word in the beginning.
Instead of (\ref{log_concave}) it obviously suffices that $\log \tilde F$ is concave to come from \eqref{3.24}
to \eqref{3.43}. Then in the sequel
we arrive analogously to (\ref{analoge_Stelle_1}) at the following inequality.
Namely,  we have in $(0, \bar t)$ that
 \begin{equation}\label{extra_3}
\begin{aligned}
0 \ge & \frac{1}{\tilde F}\tilde F^{kk}u_{11}-\log \tilde F  
-g_{11} \\
& + \gamma \{(\gamma-1)\frac{1}{\tilde F}\tilde F^{kk}u_{kk}^2-(u_i-p_i)g_i-d_0 (u+p_{n+1})\\
& -(u+p_{n+1})(u_t+p_{n+1,t})+(u_i-p_i)p_{i,t}\}Q^{\frac{\gamma-2}{2}},
\end{aligned}
\end{equation}
especially $(u_{ij})$ is assumed to be diagonal.
W.l.o.g. let us assume $u_{11}\ge ... \ge u_{nn}$ and we define the ratio $\mu =\frac{u_{11}}{u_{nn}}$. In view
of \cite[Lemma 2.2.4]{CP}, the homogeneity of $\tilde F$ and \eqref{new_condition}
we have
\begin{equation} \label{extra_1}
\frac{1}{\tilde F}\sum_k\tilde F^{kk}u_{11}=\sum_k \frac{1}{\tilde F}\tilde F^{kk}u_{11}
\ge  \mu \frac{1}{\tilde F}\tilde F^{nn}u_{nn}\ge d_0 \frac{\mu}{n}\left(\frac{1}{\mu}\right)^{\eta}\ge c \mu^{1-\eta}
\end{equation}
where we do not sum over $n$ (in the last but one expression)
and
\begin{equation}\label{extra_2}
 \sum_k\frac{1}{\tilde F}\tilde F^{kk}u_{kk}^2 \ge  \frac{1}{\tilde F}\sum_k\tilde F^{kk}u_{kk}u_{nn} =
 d_0 \frac{u_{11}}{\mu}.
\end{equation}
From (\ref{2.2}) and (\ref{2.3}) we deduce (we derived that already in the previous proof) that $|u+p_{n+1}|\le cD$ and $|u_i-p_i|\le cD$. Putting (\ref{extra_3}), (\ref{extra_1}) and (\ref{extra_2}) 
together we obtain 
\begin{equation}
\mu^{1-\eta} + (c D)^{\gamma-2} \frac{u_{11}}{\mu}
\le \log \tilde F +c + cQ^{\frac{\gamma-2}{2}}D(1+|u_t|+|H_t|)
\end{equation}
and hence
\begin{equation}
(c D)^{2-\gamma} \mu^{1-\eta} +\frac{u_{11}}{\mu} \le cD^{2-\gamma}\log u_{11} + cD^{2-\gamma}+cD(1+\log u_{11}).
\end{equation}
From the last but one inequality we get by taking the power $\frac{1}{1-\eta}$ that
\beq
\begin{aligned}
\mu \le& \{\log \tilde F +c + cQ^{\frac{\gamma-2}{2}}D(1+|u_t|+|H_t|)\}^{\frac{1}{1-\eta}} \\
\le& c \{\log \tilde F \}^{\frac{1}{1-\eta}} +c \{Q^{\frac{\gamma-2}{2}}D(1+|u_t|+|H_t|)\}^{\frac{1}{1-\eta}} + c.
\end{aligned}
\eeq
We conclude that
\beq
u_{11} \le c\left(c+ D^{c_1}(\log u_{11})^{c_2}\right)
\eeq
with some positive constants $c_1, c_2$. This implies the claim.
\end{proof}
Furthermore, we have the following analoga with proofs as before. 

\begin{cor} \label{new_korollar}
We assume case (ii) of Assumption \ref{ass1}. There exist $c_1, c_2>0$ such that
\begin{equation}
r(t) \ge \frac{c_1 R(t)^2}{1+\sup_{\tau\le t}(R(\tau))^{c_2}}.
\end{equation}
\end{cor}

\begin{lem} \label{new_lemma_2}
We assume case (ii) of Assumption \ref{ass1}. There exists a constant $c$ depending only on $n$, $r_0$, $R_0$, $f$ and initial data such that
\begin{equation}
\inf \{H_t(x,t):(x,t)\in S^n \times [0, T]\} \ge -c.
\end{equation}
\end{lem}
\begin{proof}
The proof is as before where we can now use in \eqref{fkk} the concavity and homogeneity of degree one to deduce via \cite[Lemma 2.2.19]{CP}
that
\beq
\sum_k \tilde F^{kk}\ge \tilde F(1, ..., 1).
\eeq
The remaining argument is then similar.
\end{proof}

Using a comparison principle and comparing the flow (\ref{flow_H}) with the ODE
\begin{equation}
\frac{\partial \rho}{\partial t} = \log \left(\frac{\rho^{d_0}}{F(1, ..., 1)}\right)M, \quad \rho(0)=\rho_0,
\end{equation}
where $M= \max\{f(x): x \in S^n\}$ and $\rho_0$ sufficiently large, we obtain that $H(x,t)$ is bounded in any finite time interval. Furthermore, its gradient is also bounded by (\ref{gradient_estimate}). From Krylov-Safonov estimates and parabolic regularity theory, cf. \cite{Krylov}, one gets that problem (\ref{flow_H}) has for $H_{\Theta}\in C^{4+\al}(S^n)$ a unique $C^{4+\al, 2+\frac{\al}{2}}$
solution in a maximal interval $[0, T^{*})$, $T^{*}\le \infty$ and since $H_{\Theta}$ is even of class $C^{\infty}$ in our case that this solution is also of class $C^{\infty}$. For the outer radius $R(t)$ of $X(\cdot, t)$
we have
\begin{equation} \label{outer_radius}
\lim_{t \uparrow T^{*}}R(t) = 0 
\end{equation}
if $T^{*}$ is finite.

\section{Proof of Theorem \ref{new_main_result} under Assumption \ref{ass1} (i)}\label{proof_of_main_result}
\begin{proof}[Proof of Theorem \ref{new_main_result} (i) and (ii) in the case (i) of Assumption \ref{ass1}]
(i)
We follow  the proof of \cite[Theorem A]{1} but use different arguments to deduce convergence to a translating solution. 
Let $m=\inf_{S^n}f$ and $M=\sup_{S^n}f$. If the initial hypersurface $X_{\Theta}$ is a sphere of radius $\rho_0>\left(\frac{F(1, ..., 1)}{m}\right)^{\frac{1}{d_0}}$, the solution $X(\cdot, t)$ to the equation
\begin{equation}
\frac{\partial X}{\partial t} = -\log \frac{F}{m}\nu, \quad X(\cdot,0) = X_{\Theta},
\end{equation}
remains to be spheres and the flow expands to infinity as $t\rightarrow \infty$. On the other hand, if $X_{\Theta}$
is a sphere of radius less than $\left(\frac{F(1, ..., 1)}{M}\right)^{\frac{1}{d_0}}$, the solution to
\begin{equation}
\frac{\partial X}{\partial t} = -\log \frac{F}{M}\nu, \quad X(\cdot,0) = X_{\Theta},
\end{equation}
is a family of spheres which shrinks to a point in finite time. By the comparison principle and Remark \ref{10} the solution $X(x,t)$ of (\ref{1}) will shrink to a point if $\Theta$ is small enough, and will expand to infinity if $\Theta>0$ is large.

Hence using  Corollary \ref{korollar} we obtain that the sets
\begin{equation} \label{definition_set_B}
\begin{aligned}
A =& \{\Theta>0: X(\cdot, t) \text{ shrinks to a point in finite time}\} \\
B =& \{\Theta>0: X(\cdot, t) \text{ expands to infinity as } t \rightarrow \infty\}
\end{aligned}
\end{equation}
are non-empty and open since the solution $X(x,t)$ of (\ref{1}) on a fixed finite time interval $[0, T)$ depends continuously on $\Theta$. We define
\begin{equation}
\Theta_{*} = \sup A
\end{equation}
and
\begin{equation}
\Theta^{*} = \inf B.
\end{equation}
and deduce $\Theta_{*} \le \Theta^{*}$ from the comparison principle.

Using Corollary \ref{korollar} we deduce that for any $\Theta \in [\Theta_{*}, \Theta^{*}]$ the inner radii of $X(\cdot, t)$ have a uniform positive lower bound and the outer radii are uniformly bounded from above, furthermore, $T^{*}=\infty$ in view of (\ref{outer_radius}). 
Hence (\ref{flow_H}) is uniformly parabolic and we have uniform 
bounds for $D_t^kD_x^lX(\cdot, \cdot)$ if $k+l \ge 1$, $k \ge 0$  and $l \ge 0$ on $S^n\times [0, \infty)$. 

(ii) Let $\Theta \in [\Theta_{*}, \Theta^{*}]$.
We shall use a method from \cite{4}
to show that our solution that exists for all positive times
converges to a translating solution. The main difference from our case to
\cite{4} is that we argue on the level of a derivative of 
the support function  while \cite{4} uses a
graphical representation of the flow hypersurfaces. 

One easily checks that a family of smoothly evolving uniformly
convex hypersurfaces represented by its family of support functions $\tilde H(\cdot, t)$ 
is translating 
iff there is $\xi \in \mathbb{R}^{n+1}$
so that
\begin{equation}\label{translating_sphere}
 \tilde H(x, t) = \tilde H(x, 0) + t\xi x, \quad x \in \mathbb{R}^{n+1}.
\end{equation}
Let us fix $1 \le \gamma \le n+1$ and let $e_{\gamma}$ denote the corresponding standard basis vector. 
Differentiating the 
homogeneous degree one extension (not relabeled) of (\ref{translating_sphere}) with respect to $x$
in direction $e_{\gamma}$ we get
\begin{equation} \label{translating_scalar}
\frac{\partial}{\partial x^{\gamma}}\tilde H(x,t) =  \frac{\partial}{\partial x^{\gamma}}
\tilde H(x, 0)+t\xi_{\gamma}. 
\end{equation}
Hence $\frac{\partial}{\partial x^{\gamma}}\tilde H(\cdot, t)$ is a scalar translating function. 
Conversely, if (\ref{translating_scalar}) holds then $\tilde H$ satisfies (\ref{translating_sphere}).
Note, that $\tilde H(0,t)=0$ and that $\frac{\partial}{\partial x^{\gamma} }\tilde H(\cdot, t)$ is 
homogeneous of degree zero.

Let $H$ be a solution of (\ref{flow_H}). We denote the homogeneous degree one extension of 
$H$ to $\mathbb{R}^{n+1}$ again by $H$ and the homogeneous degree 0 extension of $f$ to 
$\mathbb{R}^{n+1}\setminus\{0\}$ also by $f$. We recall the flow equation for $H$
\begin{equation} \label{recall_flow_H}
 \frac{\partial H}{\partial t} = \log \tilde Ff \quad \text{in} \quad S^n\times [0, \infty),
\end{equation}
where $\tilde F = \tilde F(r_i)$ and $r_i$, $i=1, ..., n$, are the principal radii of $M(t)$ given as non-zero eigenvalues 
of the Hessian matrix 
$\left(\frac{\partial^2H}{\partial x_{\al}\partial x_{\be}}\right)_{\alpha, \beta=1, ..., n+1}$.
Using the homogeneity of $H$ this can be rewritten 
as a flow equation for $H$ on $\left(\mathbb{R}^{n+1}\setminus \{0\}\right) \times [0, \infty)$
\begin{equation} \label{rewritten_flow}
\begin{aligned}
\frac{\partial H}{\partial t}(x,t) = & |x|\frac{\partial H}{\partial t}\left(\frac{x}{|x|},t\right) \\
=  & |x| \log \tilde F f
\end{aligned}
\end{equation} 
where $\tilde F = \tilde F(r_i)$ and $r_i$, $i=1, ..., n$, are the principal radii of $M(t)$ given as non-zero eigenvalues 
of the matrix 
$\left(|x|\frac{\partial^2H}{\partial x_{\al}\partial x_{\be}}\right)_{\alpha, \beta=1, ..., n+1}$ at $\left(x, t\right)$ and $f=f(x)$. 
We will replace (formally) the curvature function $\tilde F$ in Equation (\ref{rewritten_flow}) by a curvature function $\hat F$ which depends
on all eigenvalues $r_{\alpha}$, $\alpha=1, ..., n+1$, of $\left(|x|\frac{\partial^2H}{\partial x_{\al}\partial x_{\be}}\right)_{\alpha, \beta=1, ..., n+1}$ at $\left(x, t\right)$ and satisfies $\tilde F(r_i) = \hat F(r_{\alpha})$
in order to be notational in the framework of the introduction.

a) In the case that $\tilde F \in C^{\infty}(\bar \Gamma_+)$ and $\tilde F_{|\partial\Gamma_+}=0$
we define
\begin{equation}
 \hat F(r_1, ..., r_{n+1}) = \sum_{\alpha_0=1}^{n+1}\tilde F(\hat r^{\alpha_0})
\end{equation}
where $\hat r^{\alpha_0} = (r_1, ..., r_{\alpha_0-1}, r_{\alpha_0+1}, ..., r_{n+1})$.

b) Let us consider the general case (which includes case a)). In view of our a priori estimates there are constants $b_1, b_2>0$ so that the non-zero eigenvalues of $\left(\frac{\partial^2H}{\partial x_{\al}\partial x_{\be}}\right)_{\alpha, \beta=1, ..., n+1}$ on $S^n \times [0, \infty)$ are in the interval $[b_1, b_2]$.  Having the later application of the argumentation in \cite[Subsection 6.2]{4} in mind we remark that this property carries over to the Hessians of convex combinations of $H(\cdot, t_1)$ and $H(\cdot, t_2)$ with arbitrary $t_1, t_2>0$. Note that the vector $x$ is a zero eigenvector  of the Hessian of $H$ at every $(x,t)\in S^n\times [0, \infty)$. We define
 \begin{equation}
 \hat F(r_1, ...,  r_{n+1}) = \tilde F(\hat r) + \check r
\end{equation}
on the set
\begin{equation}
\Omega = \bigcup_{1 \le \alpha \le n+1} I_{\alpha}
\end{equation}
where
\begin{equation}
I_{\alpha} =  \left(\frac{b_1}{2}, \infty\right)\times ... \times \left(\frac{b_1}{2}, \infty\right) \times  \left(-\frac{b_1}{2}, \frac{b_1}{2}\right) \times \left(\frac{b_1}{2}, \infty\right) \times ... \times  \left(\frac{b_1}{2}, \infty \right)
\end{equation}
with factor $(-\frac{b_1}{2}, \frac{b_1}{2})$ at position $\alpha$ and 
where $\check r=r_{\alpha_0}=\min_{\alpha=1, ..., n+1}r_{\alpha}$, $\alpha_0 \in \{1, ..., n+1\}$ suitable, and $\hat r=(r_1, ... r_{\alpha_0-1}, r_{\alpha_0+1}, ..., \alpha_{n+1})$.
We have
$\tilde F(r_i) = \hat F(r_{\alpha})$. 
From standard arguments we deduce that $\hat F$ defines in the way explained in the introduction a differentiable function on the set of symmetric matrices with eigenvalues in $\Omega$. 

Differentiating (\ref{rewritten_flow}) we get the following equation for $H_{\gamma}$
\begin{equation} \label{translating}
 \frac{\partial}{\partial t}H_{\gamma}(x,t) =|x|^2\frac{1}{\hat F}\hat F^{\al\be}
 \left(H_{\ga}\right)_{\al\be}+|x|_{\gamma} \log \hat F f +|x|\frac{f_{\ga}}{f}+d_0|x|_{\gamma}
\end{equation}
where $\hat F^{\al\be}$ is uniformly elliptic and the coefficients of the elliptic operator on the right-hand side depend on the derivative of $H_{\gamma}$ and $x$ and not explicitly on $t$ or  $H_{\gamma}$.

Applying the argumentation from \cite[Subsection 6.2]{4} more or less word by word to the function $H_{\gamma}$ on $\left(B_{\rho_2}(0)\setminus B_{\rho_1}(0)\right)\times [0, \infty)$, $0< \rho_1 < 1< \rho_2$ both close to 1, where we use that $H_{\gamma}$ is homogeneous of degree zero (instead of the compactness of the spatial domain  and the boundary condition when we apply maximum principles) we obtain that $H_{\gamma}$ converges smoothly to a translating solution of (\ref{translating}) with a translating speed $\xi=\xi(\Theta, \gamma)\in \mathbb{R}$.

(iii) We show $\Theta_{*}=\Theta^{*}$.  From (ii) we know that for every $\Theta \in [\Theta_{*}, \Theta^{*}]$ the solution $X(x,t)$ of (\ref{1}) with initial value $X_{\Theta}$ converges to a translating solution with a certain translating speed $\xi_{\Theta}\in \mathbb{R}^{n+1}$. 

a) We show that there is $\xi \in \mathbb{R}^{n+1}$ so that $\xi_{\Theta}=\xi$ for all $\Theta \in [\Theta_{*}, \Theta^{*}]$. For this let
$\Theta_{*} \le \Theta_1 < \Theta_2 \le \Theta^{*}$,  differentiating (\ref{flow_H}) in $\Theta$ gives
\begin{equation}
\begin{aligned}
\frac{\partial H'}{\partial t} =& A^{ij}(\nabla_{i}\nabla_{j}H'+H'\delta_{ij}) \\
H'(0)=&\frac{d}{d\Theta}H_{\Theta}
\end{aligned}
\end{equation}
where $(A^{ij})$ is the inverse of $(\nabla_{i}\nabla_{j}H+\de_{ij}H)$. By the maximum principle
\begin{equation}
H'(x,t) \ge \min_{S^n} \frac{d}{d\Theta}H_{\Theta}(x).
\end{equation}
Thus 
\begin{equation} \label{ungleichung}
\begin{aligned}
c(x,t)+t(\xi_{\Theta_2}-\xi_{\Theta_{1}})x =&H_{\Theta_2}(x, t)-H_{\Theta_1}(x, t) \\
\ge& \int_{\Theta_{1}}^{\Theta_2}\min_{S^n} \frac{d}{d\Theta}H_{\Theta} >0
\end{aligned}
\end{equation}
where $c(x,t)$ is a uniformly bounded function and where we used Lemma \ref{11}. This implies $\xi_{\Theta_1}=\xi_{\Theta_2}$.

b) Using a) we deduce from the comparison principle that $H_{*}=H^{*}$ where $H_{*}$ and $H^{*}$ is the solution of $F=e^{\xi x}f$
starting from $H_{\Theta_{*}}$ and $H_{\Theta^{*}}$, respectively.
We deduce from (\ref{ungleichung}) with $\Theta_1=\Theta_{*}$ and $\Theta_2=\Theta^{*}$ by using that $H_{\Theta_2}(\cdot, t)-H_{\Theta_1}(\cdot, t)$ converges uniformly to zero as $t\rightarrow \infty$ that $\Theta_{*}<\Theta^{*}$ leads to a contradiction, hence $\Theta_{*}=\Theta^{*}$.

(iv) We show that the normalized hypersurface $X(\cdot, t)/r(t)$ converges to a unit sphere in case $\Theta>\Theta^{*}$ and follow for it the lines of \cite[Theorem B]{1}.
 Since $X$ is expanding, we may w.l.o.g. assume at $t=0$ that it contains the ball $B_{R_1}(0)$ where $R_1>1+\left(\frac{F(1, ..., 1)}{m}\right)^{\frac{1}{d_0}}$, $m=\inf_{S^n}f$, and  that it is contained in the ball  $B_{R_2}(0)$ where $R_2>0$ is sufficiently large. 
For $i=1,2$ let $X_i(\cdot, t)$ be the solution of (\ref{1}) where $f$ is replaced by $m$ and $M=\sup_{S^n}f$ respectively and $X_i(\cdot, 0)=\partial B_{R_i}$. The $X_i(\cdot, t)$ are spheres and their radii $R_i(t)$ satisfy
\begin{equation} \label{ungleichung3}
c^{-1}(1+t)\log(1+t) \le R_1(t) \le R_2(t) \le c(1+(1+t)\log^2(1+t))
\end{equation}
for some $c>0$. We deduce from the ODEs for the $R_i$ , $i=1, 2$, that
\begin{equation}
\begin{aligned}
\frac{d}{dt} (R_2(t)-R_1(t)) \le& d_0 \log \frac{R_2(t)}{R_1(t)} + c \\
\le& c \log \log (1+t) + c
\end{aligned}
\end{equation}
where the last inequality uses (\ref{ungleichung3})
and hence
\begin{equation}
R_2(t)-R_1(t) \le c (1+ t \log \log(1+t))
\end{equation}
so that
\begin{equation}
\lim_{t\rightarrow \infty}\frac{R_2(t)-R_1(t)}{R_1(t)}=0. 
\end{equation}
By the comparison principle $X(\cdot, t)$ is pinched between $X_2(\cdot, t)$ and $X_1(\cdot, t)$ 
and, furthermore, we deduce that $X(\cdot, t)/r(t)$ converges to the unit sphere uniformly.

The proof of Theorem \ref{new_main_result} (i) and (ii) is finished in the case (i) of Assumption \ref{ass1}.
\end{proof}



\section{Proof of Theorem \ref{new_main_result} in the remaining cases via diameter bound}\label{additional_section}
Throughout this section (with an exception in a short passage below which we indicate explicitly) we assume that case (ii) of Assumption \ref{ass1} holds and prove Theorem
\ref{new_main_result} in this case by presenting (only) the arising differences
to the proof of Theorem
\ref{new_main_result} (i) and (ii) under Assumption \ref{ass1} (i) in the previous section. 
The crucial difference is that 
the set $B$, cf. (\ref{definition_set_B}), has now to be redefined in view of Corollary \ref{new_korollar} 
(which states only a poor lower bound for the inradii) and that its redefined version 
cannot be identified as open immediately, so further work is necessary.
Deviating from (\ref{definition_set_B}) as far as $B$ is concerned
we define the intervals $A$  and $B$ (intervals due to a comparison principle) now as
\begin{equation} \label{definition_interval}
\begin{aligned}
A =& \{\Theta>0: X(\cdot, t) \text{ shrinks to a point in finite time}\} \\
B =& \{\Theta>0: \diam X(\cdot, t) \text{ converges to infinity as } t \rightarrow \infty\}
\end{aligned}
\end{equation}
and $\Theta_{*}=\sup A$ and $\Theta^{*}=\inf B$. 
 Similarly as in the proof of part (i) of Theorem \ref{new_main_result} under Assumption \ref{ass1} (i) 
 one obtains that $[\Theta_{*}, a_1]$ is empty whenever $\Theta_{*}<a_1<\Theta^{*}$ and hence $\Theta_{*}=\Theta^{*}$.
 Clearly, by comparing with spheres, $A$ is open.   
 The openness of $B$ which is a priori not clear
 under our present assumptions given by case (ii) of Assumption \ref{ass1} follows from the following Lemma \ref{extra_lemma}
 which will be proven by using geometric observations. Once this openness is established following the 
 lines of the proof in the previous section we obtain 
 convergence to a translating solution which translates with a certain speed $\xi\in \mathbb{R}^{n+1}$ and with
 $F$-curvature when considered as a function of the normal given by 
 \begin{equation} \label{representation_curvature}
 e^{\xi\cdot x}f(x), \quad x\in S^n,
 \end{equation} 
 where $\xi\in \mathbb{R}^n$ fixed. Hence, inclusively Lemma \ref{extra_lemma}, the proofs of 
 Theorem \ref{new_main_result} (i) and (ii) are complete.
 
 
 
 The following lemma states the crucial diameter bound which holds uniformly for all $t\in [0, \infty)$.
\begin{lem}\label{extra_lemma}
Let $\Theta=\Theta_{*}$ and $X(\cdot, t)$ (also denoted by $M(t)$) be the flow hypersurfaces of the flow
\beq\label{flow1}
\frac{\partial X}{\partial t}=-\log \left(\frac{F}{f}\right)\nu, \quad X(S^n, 0)=M_{\Theta},
\eeq 
then there is $c>0$ such that
\beq \label{diameter_inequality}
\diam X(\cdot, t) \le c
\eeq
for all $t \in [0, \infty)$, especially the flow 
exists for those times.
\end{lem}

{\it Picture to the proof:} 
The idea of the proof is to assume the contrary and to construct a barrier for the flow from a certain time $T>0$ on, where $T$ is appropriately chosen. The barrier consists of a cylinder and two spheres. All these three objects evolve according to \eqref{flow1} but are regularly artificially
 adjusted so that they always form roughly spoken a long dumbbell and, furthermore,  so that this dumbbell encloses the flow hypersurface at every time.
 Then we argue that we can arrange everything so that the diameter of this dumbbell decreases at least with a sufficiently large  positive average speed for a sufficiently large period of time which leads to the desired contradiction, see also Figure \ref{setup}. 

\begin{figure}[t]
\centering
\includegraphics[width=7.3cm]{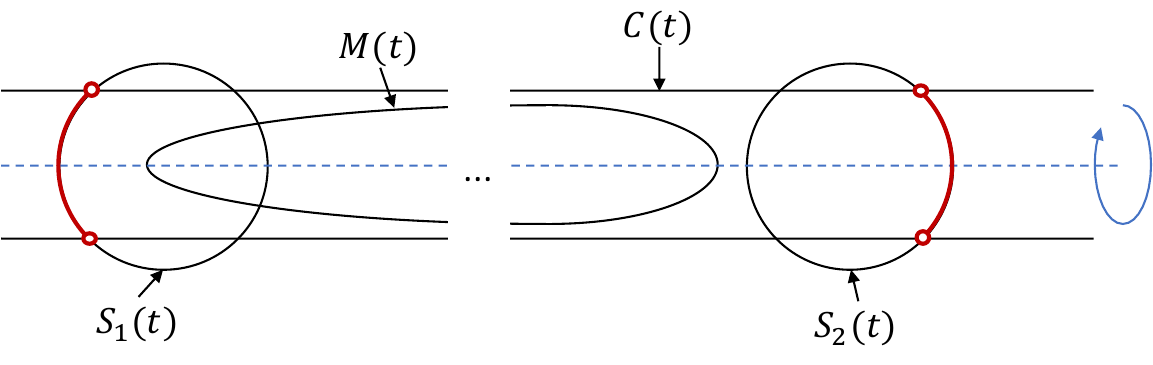}
\caption{Section with the $y=0$ plane of the rotationally symmetric flow hypersurface $M(t)$ and the barrier consisting of two spheres $S_i(t)$, $i=1,2$, and a cylinder $C(t)$. All objects are rotationally symmetric with respect to the same dashed horizontal $x$-axis. The three dots indicate that the middle part of $M(t)$ (not plotted) is very long compared to the plotted ends of $M(t)$. 
All four objects evolve under the same flow but the spheres and the cylinder are regularly artificially adjusted. Thereby, especially
the diameters of the spheres are always (only) slightly larger than the cross-section diameter of the cylinder.
 } \label{setup}
\end{figure}
 
\begin{proof}[Proof of Lemma \ref{extra_lemma}]
Before we begin with the indirect proof below we make some preparatory comments on the proof and notation.
The proof uses several lemmas which inclusive their proofs will be interspersed into the running text proving Lemma \ref{extra_lemma}.
We present the proof of Lemma \ref{extra_lemma} only in the case $n=2$. The general case can be treated analogously. For a subset $S \subset \mathbb{R}^{n+1}$
we denote the set of its inner points by $\inte S$ and its convex hull by $\co S$.


We make the following observation.
Let $R>0$ so that
\beq
M_{\Theta}\subset B_{R-1}(0)
\eeq
and $\tilde M(t)$ be the flow hypersurfaces of the flow \eqref{flow1} starting from the sphere $\partial B_R(0)$ at time $t=0$ with $f$ in the flow speed \eqref{flow1}  replaced by $\max_{S^n}f$.
In view of a comparison principle similarly
as in part (iv) of the proof of Theorem \ref{new_main_result} we conclude 
 that the $\tilde M(t)$ enclose $M(t)$ as long as both flows exist.
 From the ODE for the radii, cf. \eqref{ungleichung3}, of the $\tilde M(t)$ we get that for all $t_0\in (0, \infty)$ the following holds: If the flow \eqref{flow1} exists for all $t \in [0, t_0)$ then 
there exists $c_0=c_0(t_0)>0$ such that 
\beq \label{c_0_upper_bound}
\diam X(\cdot, t) \le c_0 
\eeq
for all $t \in [0, t_0)$. Hence by definition of $\Theta_{*}$ (excluding that $\diam X(\cdot, t)$ goes to zero in finite time) and due to our a priori estimates given by Lemma \ref{new_lemma_1}, Corollary \ref{new_korollar} and
Lemma \ref{new_lemma_2} as well as Equation (\ref{flow_H}) 
also uniform higher order estimates of the flow hold for all $t \in [0, t_0)$. This implies existence of the flow for all $t\in [0, \infty)$.
The crucial point which remains to show is that $c_0$ in \eqref{c_0_upper_bound} can be chosen the same for all
$t\in [0, \infty)$.  Hereby, we conclude the introductory remarks  and begin with the indirect argumentation.

This means that we {\it argue 
indirectly}, set $d(t)=\diam X(\cdot, t)$, $t>0$,  and assume the following:
\bea \label{contraposition}
& \text {{\it  There is a sequence of times} }(t_k)_{k \in \mathbb{N}}, t_k>0, t_k\rightarrow \infty \text{ {\it as} } k\rightarrow \infty  \\
& \text{ \it {such that} }
d(t_k)\rightarrow \infty \text{ {\it as} } k\rightarrow \infty.
\eea
We introduce some notation. 
We fix an Euclidean coordinate system $(x,y,z)$ so that the $x$-axis is the axis of the rotational symmetry of $M_{\Theta}$. (The $x$-axis is unique up to translation, scaling 
and orientation and after fixing a choice for the $x$-axis the $y$-axis and $z$-axis are unique up to rotations around the $x$-axis and orientation and we fix here also a certain choice.)
Since the rotational symmetry is preserved under the flow \eqref{flow1} all $M(t)$, $t\ge 0$, are rotationally symmetric with respect to the $x$-axis.
Since $M(t)$, $t \ge 0$, is closed, uniformly convex and rotationally symmetric with respect to the $x$-axis there are exactly two intersection points of $M(t)$ with the $x$-axis
and their coordinates can be written as
\beq \label{anfangende}
(p(t), 0, 0), \quad (q(t), 0, 0)
\eeq
with suitable $p(t), q(t) \in \mathbb{R}$, $p(t)<q(t)$.
For $t \ge 0$ we denote the width of $M(t)$ by
\beq
\width(t) := q(t)-p(t),
\eeq
the inradius of $M(t)$ by $r(t)$, the smallest principal curvature of $M(t)$ in $p\in M(t)$ by $\kappa_{min}(t,p)$ and 
the largest principal curvature of $M(t)$ in $p\in M(t)$ by $\kappa_{max}(t,p)$. We furthermore set
\beq
\kappa_{min}(t) = \min_{p\in M(t)}\kappa_{min}(t,p)
\eeq
and 
\beq
\kappa_{max}(t) = \max_{p\in M(t)}\kappa_{max}(t,p)
\eeq
for $t\ge 0$.
We summarize for convenience some relevant facts in a remark.

\begin{rem}
In view of Lemma \ref{new_lemma_1}, Corollary \ref{new_korollar},  Lemma \ref{new_lemma_2} and Equation (\ref{flow_H})  as well as the definition of $\Theta_{*}$ the functions
$\kappa_{min}(t)$, $\kappa_{max}(t)$ and $r(t)$ are positive for $t\in [0, \infty)$. Clearly, they are also continuous.
Note that there are no constant positive lower bounds for $\kappa_{min}(t)$ and $r(t)$ available and also no constant upper bound for 
$\kappa_{max}(t)$  which hold uniformly for all $t\in [0, \infty)$. 
\end{rem}

\begin{lem} \label{claim1}
 There is $d_0>0$ so that for all $t>0$ we have
\begin{equation}
 M(t) \subset C_{d_0}=\{(x,y,z)\in \mathbb{R}^3: |y|^2+|z|^2\le d_0^2\}.
\end{equation}
\end{lem}

\begin{proof}
We prove Lemma \ref{claim1} indirectly and assume for it the contrary. Then for arbitrary large $d_1>0$
we find $t_0>0$ and $x_0\in \mathbb{R}$ such that
\beq \label{slice_1}
\circl(x_0, d_1):=\{ (x_0, d_1\cos \theta, d_1 \sin \theta):\theta \in [0, 2\pi]\} \subset 
\inte \left(\co M(t_0)\right),
\eeq
where we used that $M(t_0)$ is rotationally symmetric with respect to the $x$-axis. 
For given large $d_0>0$ we may choose $t_0>0$, $\Delta t>0$ and $d_1>d_0$ such that
\beq \label{condition_123}
\max_{M(t_0)}z=d_0, \quad \max_{M(t)}z\ge d_0\quad  \forall t \in [t_0, t_0+\Delta t], \quad \max_{M(t_0+\Delta t)}z=d_1.
\eeq
We would like to allow for $d_0$ and $\Delta t$ being chosen arbitrarily large. 
For it we consider the cylindrical flow starting from the cylinder $C_{d_0}$, $d_0>0$ large assumed, at time $t=t_0$ as outer barrier, denote its cross-section radius at time $t$ by $c(t)$ and obtain the ODE
\beq
\dot c(t) = -\log \left(\frac{F}{f}\right)= \log \left(\frac{f}{F}\right)=\log\left(\frac{fc}{F(1,0)}\right)\le \frac{fc}{F(1,0)}
\eeq
which implies that
\beq
c(t_0+\Delta t) \le e^{\frac{\max f}{F(1,0)}\Delta t}{c(t_0)},
\eeq
which is the desired quantitative relation. It tells us that for every hypothetical realization of  $\Delta t>0$ 
which makes the right-hand side 
staying below the value of $d_1$ there is some larger choice which can be taken indeed.
Note that we may here assume that
\begin{equation}
\width(t) \le c=c(f) \quad \forall t \in [t_0, t_0+\Delta t]
\end{equation}
since, otherwise, one of the $M(t)$, $t\in [t_0, t_0+\Delta t]$, encloses by convexity a large ball, contradicting the choice of 
$\Theta^{*}$.

After having fixed that preparatory setting (consisting of \eqref{condition_123} and the assumption that $\Delta t$ as well as $d_0$ are chosen sufficiently large) we give spatially averaged speed estimates on the 'approximately flat sides' of $M(t)$ as follows. 
Let $B^2_r(0)$, $1\le r << \frac{d_0}{2}$, be a two-dimensional ball around 0 in the $\{x=0\}$ plane. We define for $t\in [t_0, t_0+\Delta t]$ the set $V(t)$ as that connected component of  (the two of)
\beq
\pr_t^{-1}\left(B^2_r(0)\right) 
\eeq
where larger $x$ values are attained (clear from picture) and where
\beq
\pr_t: M(t)\rightarrow \{x=0\}
\eeq
denotes the orthogonal projection of $M(t)$ on $\{x=0\}$
 as graphs over $B^2_r(0)$. This is possible by continuity for fixed $t\in [t_0, t_0+\Delta t]$ and $r=r(t)>0$ small and from the global picture 
 we may assume $r=1$ uniformly for all $t\in [t_0, t_0+\Delta t]$ provided $d_0$ is large.
This means we write
\beq
V(t) = \graph u(t, \cdot)_{|B^2_r(0)}
\eeq
with a smooth function $u:[t_0, t_0+\Delta t]\times B^2_r(0)\rightarrow \mathbb{R}$. (This graphical representation is
understood so that for $(0, y,z)\in B^2_r(0)$ the via $u(t, \cdot)$ assigned point on $V(t)$ is $(u(t, y,z), y, z)$.)
 A simple observation shows that the convexity of the $M(t)$ implies that for given $\varepsilon>0$ we may assume that
\beq \label{smallness_derivative}
|D_yu(t, y,z)|+|D_zu(t, y,z)| \le \varepsilon \quad \text{ and } \quad |V(t)|\ge 1
\eeq 
for $(t,y,z)\in [t_0, t_0+\Delta t]\times B^2_r$ provided $d_0$ is sufficiently large. Such a property holds even without assuming rotational symmetry, here the argument is especially easy and left to the reader.

We estimate the spatial average speed of $V(t)$ in the direction of the positive $x$-axis as follows.
Using Jensen's inequality we have with $U=B^2_r(0)$
\begin{equation} \label{average_speed_graph}
 \begin{aligned}
  \frac{1}{|U|}\int_U\frac{d}{dt} u =& -\frac{|V(t)|}{|U|}\frac{1}{|V(t)|} \int_U\sqrt{1+|Du|^2}\log\left(\frac{F}{f}\right)\\
  \ge&  -\frac{|V(t)|}{|U|}\log
  \left(\frac{1}{|V(t)|} \int_U \sqrt{1+|Du|^2}\frac{F}{f}\right)\\
  \ge& -\frac{|V(t)|}{|U|}\log
  \left(\frac{1}{|V(t)|} \frac{2}{\min_{S^n}f} \int_U F \right).
 \end{aligned}
\end{equation}
In view of \cite[Lemma 2.2.20]{CP} we have that
\beq
F \le \frac{F(1,1)}{2}H
\eeq
so that we may estimate
\beq
\int_U F \le  \frac{F(1,1)}{2} \int_U H.
\eeq
Choosing $\varepsilon$ according to \eqref{smallness_derivative}, we estimate by divergence theorem that
\begin{equation}
\begin{aligned}
0<& 
 \int_UH\\
 =& \int_{U}\dive \left(\frac{Du}{\sqrt{1+|Du|^2}}\right)\\
 =& \int_{\partial U}\left<\frac{Du}{\sqrt{1+|Du|^2}}, \omega\right> \\
 <& c\varepsilon
 \end{aligned}
\end{equation}
where $\omega$ denotes the outer unit normal of $\partial U$.
This shows that for given $\eta>0$ we have
\beq
\int_U \frac{d}{dt}u \ge \eta
\eeq
provided $\varepsilon$ is sufficiently small (which can be achieved by choosing $d_0$ sufficiently large).
Integration from $t_0$ to $t_0+\Delta t$ yields
\beq
\int_Uu(t_0+\Delta t, \cdot)-\int_Uu(t_0, \cdot) \ge \eta \Delta t.
\eeq
Hence there is $(y_0, z_0)\in B^2_r(0)=U$ such that 
\beq
u(t_0+\Delta t, y_0, z_0) \ge \eta \Delta t + u(t_0, y_0, z_0).
\eeq
In view of \eqref{smallness_derivative} and assuming $\varepsilon$ small we  conclude that
\beq
u(t_0+\Delta t, y,z) \ge \frac{\eta}{2}\Delta t + u(t_0, y, z)
\eeq
for all $(y,z)\in U$. 
By symmetry we get
\beq \label{5.30}
\width(t_0+\Delta t) \ge \eta \Delta t + \width(t_0).
\eeq
Let $\tilde x \in \mathbb{R}$, $R>0$ be so that the flow $\eqref{flow_H}$ starting from $\partial B_R(\tilde x, 0, 0)$
expands to infinity.
In order to achieve \eqref{5.30} we assumed that $d_0$ is large but the previous deliberations showed that we may also assume that $\Delta t$ (and $\eta \Delta t$) is large. 
Hence we may assume that both are so large that 
up to translation we have  the inclusion
\beq
\partial B_R(\tilde x, 0, 0) \subset \inte \co M(t_0+r \Delta t)
\eeq
which contradicts the choice of $\Theta^{*}$. This proves Lemma \ref{claim1} .
\end{proof}

\begin{rem}
There is $d_1>0$ such that the maximum extension of $M(t)$ perpendicular to the $x$-axis is never below $d_1$, i.e.
\beq
\forall\ t \ge 0\  \exists\ x' \in \mathbb{R} : \quad \diam \left(M(t)\cap \{x=x'\}\right) \ge d_1.
\eeq
\end{rem}
\begin{proof}
Suppose the assertion is wrong and choose $d_1>0$ so small that the cylinder $\{(x,y,z)\in \mathbb{R}^3:y^2+z^2=(2d_1)^2\}$ contracts under flow \eqref{flow1} to a line and at the same time serves as an outer barrier of the $M(t)$. This contradicts the choice of $\Theta^{*}$.
\end{proof}

The following lemma shows that for a given time span $\Delta t>0$ there is a $\delta>0$ such that for all $t_0>0$ the following holds:
All points in $\mathbb{R}^3$ having distance from $\co M(t_0)$ larger than $\delta$ remain disjoint from $M(t)$
for all $t \in [t_0, t_0+\Delta t]$. The crucial point hereby is that $\delta$ does not depend on $t_0$. The statement trivializes immediately
when $\delta$ is allowed to depend on $\diam M(t_0)$ because then we can compare the flow with a barrier consisting of a single large sphere flowing according to \eqref{flow1} and $f$ therein replaced by $\max_{S^n} f$.

\begin{lem}\label{claim2}
For every time span  $\Delta t>0$ exists a $\delta > 0$ so that for all $t_0>0$ the following holds:
If $p\in \mathbb{R}^3$, $\dist(p, M(t_0))>\delta$ then $p\notin M(t)$ for all $t\in [t_0, t_0+\Delta t]$.
\end{lem}

\begin{proof}
We fix $t_0>0$ and construct a barrier for the flow \eqref{flow_H} starting at time $t=t_0$ from $M(t_0)$.
The claim will follow from an obvious bound for the speed of expansion of the diameter of that barrier.
The barrier consists of a cylinder $C$ and two spheres $S_1, S_2$ where
\bea
C=&\{(x,y,z):\mathbb{R}^3:y^2+z^2=(d_0+1)^2\} \\
S_1=&\partial B_{2d_0}(p(t_0)-1, 0, 0) \\
S_2=&\partial B_{2d_0}(q(t_0)+1, 0, 0)
\eea
where $d_0>0$ is chosen according to Lemma \ref{claim1}.
Furthermore, starting further flows according to \eqref{flow1} at time $t=t_0$
from $C$, $S_1$ and $S_2$, and denoting the corresponding flow surfaces by $C(t)$, $S_1(t)$ and $S_2(t)$, $t\ge t_0$, respectively, we claim that 
\beq \label{containment_relation}
M(t) \subset \co(S_1(t) \cup S_2(t))
\eeq
for $t\in [t_0, t_0+\Delta t]$ with some $\Delta t=\Delta t(C, S_1, S_2)>0$.
 Namely, choosing $\Delta t>0$
so small that the $C^1$-distance of $C(t)$ and $C$ and the $C^1$-distance of $S_i(t)$ and $S_i$, $i=1,2$,
is not too large for $t \in [t_0, t_0+\Delta t]$, we may conclude from the existence of a minimal
$t_1 \in (t_0, t_0+\Delta t]$ with
\beq
M(t_1) \cap \partial \left(\co(S_1(t) \cup S_2(t))\right) \neq \emptyset
\eeq
that for some minimal $\tilde t\in (t_0, t_1]$ we may choose
\beq \label{case2}
p \in M(\tilde t )\cap C(\tilde t)\neq \emptyset,
\eeq 
or that there is (w.l.o.g., otherwise take $S_2$ instead of $S_1$)
\beq \label{case1}
p \in M(\tilde t)\cap S_1(\tilde t),
\eeq
 $\tilde t=t_1$,  in which both intersecting surfaces have (in both cases) a common outer unit normal. 
In both cases there is an open neighborhood $U$ of $p$ in $\mathbb{R}^{n+1}$ and a time $\Delta \tilde t>0$ so that those parts of both intersecting flows
which are  in $U$ for $t\in [\tilde t-\Delta \tilde t, \tilde t]$ are for that times graphs over their common tangent plane $T_pM(\tilde t)$ in $p$ at time $\tilde t$. Now we apply Lemma \ref{maximum_principle} to conclude that in both cases \eqref{case2}  and \eqref{case1}
the two flows in $U$ that intersect at time $\tilde t$ already intersect at an earlier time $\hat t < \tilde t$ which is a contradiction to the 
choice of $\tilde t$.
This proves \eqref{containment_relation}.

Let us interprete  \eqref{containment_relation} in other words. It  gives us full quantitative control in the sense of an upper bound for the extension 
of $M(t)$ for $t\in [t_0, t_0+\Delta t]$ which depends only on $d_0$, especially it bounds
\beq
p_{new}=\inf_{t\in [t_0, t_0+\Delta \tilde t]}p(t)
\eeq
from below and
\beq
q_{new}=\sup_{t\in [t_0, t_0+\Delta t]}q(t)
\eeq
from above. Together with Lemma \ref{claim1} this controls the extension of $M(t)$ for $t\in [t_0, t_0+\Delta t]$. Now we iterate the whole argument, this means for the first repetition 
that we replace $p(t_0)$, $q(t_0)$ and $M(t_0)$ by
$p_{new}$, $q_{new}$ and $M(t_0+\Delta t)$, respectively, and get
 a quantitative estimate for the extension of $M(t_0+2\Delta t)$ and updated values for $p_{new}$ and $q_{new}$.
Note that $d_0$ and hence also $\Delta t$ are always the same in each iteration. 
Furthermore, the values of $p_{new}$ and $q_{new}$ change in each iteration step at most by a uniform  additive constant
which  depends only on $d_0$.
Note that $\Delta t$ together with the additive constant trivially define an upper bound for a 'speed of extension of $M(t)$, timely averaged over $[t_0, t_0+\Delta t]$', hence the claim of the lemma follows.
\end{proof}

\begin{ass}  \label{ass_long_extension}
 For the rest of the proof of Lemma \ref{extra_lemma} we fix a time interval $I=[t_0, t_1]$ with $0<t_0<t_1<\infty$ and assume by using \eqref{contraposition} that
 \begin{enumerate} 
 \item $\Lambda_0:=d(t_0)\le d(t)$
for $t\in I$, 

\item  $|I|$ is large and $\Lambda_0$ is large compared to $|I|$ which is possible in view of Lemma \ref{claim2}.

\item  Furthermore, we may assume w.l.o.g. that $p(t_0)=0$, meaning that the surface at time $t_0$ contains the origin of the coordinate system  and that it is contained completely in the $\{x\ge 0\}$ half space.
For $t\in I$ and $x'\in \mathbb{R}$ we denote sections of the flow surfaces $M(t)$ perpendicular to the $x$-axis by $C(t,x')=\{x=x'\}\cap M(t)$ and may assume that $C(t, \frac{1}{10}\Lambda_0)$ and 
$C(t, \frac{9}{10}\Lambda_0)$ are both nonempty for all $t \in I$. Note that this is possible in view of Item (1), Lemma \ref{claim2}, the connectedness of $M(t)$, $t\in I$, and since $\Lambda_0$ is large compared to $|I|$.
\end{enumerate}
\end{ass}

For the proof of the following lemma we will use  refined but similar estimates for averaged speeds as in the proof of Lemma \ref{claim1}.

\begin{lem}\label{claim3} There is $\bar \delta>0$ which can be chosen small for $\Lambda_0$ sufficiently large so that
\begin{equation} \label{beh1}
 \diam C(t,x)\le \frac{2+\bar \delta}{f(0,0,1)}\frac{F(1, ..., 1)}{2}
\end{equation}
for all $t \in \tilde I=[t_0, \frac{t_0+t_1}{2}]$ and all $x\in [\frac{2}{10}\Lambda_0,\frac{8}{10}\Lambda_0]$.
\end{lem}

\begin{proof}
We prove Lemma \ref{claim3} indirectly. For it we assume that $\Lambda_0>0$ is large, $\bar \delta>0$ small and that there are $\tilde t\in \tilde I$, $x'\in [\frac{2}{10}\Lambda_0,\frac{8}{10}\Lambda_0]$
so that (\ref{beh1}) does not hold for $t=\tilde t$ and $x=x'$. W.l.o.g. we may assume that 
 (\ref{beh1}) does not hold for $t=\tilde t$ and
$x\in (x'-w, x'+w)$ where $0<w<1$ is a fixed width. We 
consider the surface parts
\begin{equation}
 R=R(t, x') = \{x'-w\le x\le x'+w\}\cap M(t), \quad t \in I.
\end{equation}
We introduce some notation. For $t\in \tilde I$, $p =(p_x,p_y,p_z)\in R(t, x')$, we denote 
the outer unit normal  vector of $M(t)$ in $p$ by $\nu(t,p)$ and the unit radial vector 
by $\tilde \nu(p)=(0, \frac{p_y}{l}, \frac{p_z}{l})$ where $l=(p_y^2+p_z^2)^{\frac{1}{2}}$. In view of Lemma \ref{claim1}
and since $\Lambda_0$ is assumed to be large we have by using the convexity of $M(t)$ that
\beq \label{normals_similar}
\left<\tilde \nu(p), \nu(t,p)\right> \ge 1-\delta,
\eeq
where $\delta=\delta(\Lambda_0, d_0)$ is a positive constant which can be assumed to be arbitrarily
small if $\Lambda_0$ is sufficiently large.
In view of \eqref{normals_similar} and standard arguments using the implicit function theorem we may write $R(t, x')$ as graph 
in cylindrical coordinates around the $x$-axis which means the following 
\beq
R(t) = \{(x, u(t,x) \cos \theta, u(t,x)\sin \theta ): x \in (x'-w, x'+w), \theta \in [0, 2\pi]\}
\eeq
where
\beq
u(t, \cdot) \in C^{\infty}((x'-w, x'+w)), u\in C^0(\tilde I \times (x'-w, x'+w)).
\eeq
By using e.g. a first order expansion of the flow $X$ in a neighborhood around $(t,p)$
with
\beq
p=(x, u(t,x) \cos \theta, u(t,x) \sin \theta)
\eeq
so that this neighborhood contains $(t+h, p(h))$ with
\beq
p(h)=(x, u(t+h,x) \cos \theta, u(t+h,x) \sin \theta), 
\eeq
$h>0$ small, 
we conclude that the speed in outward radial direction (calculated as limit of a difference quotient) of the graph is given by
\beq
\frac{\partial u}{\partial t}(t,x) = -\log \left(\frac{F}{f}\right)\left<\nu(t,p), \tilde \nu(t,p)\right>
\eeq
where
 $H$ and $f$ are evaluated on $M(t)$ in $p$.
Especially, $\frac{\partial u}{\partial t}$ is
a continuous function in $\tilde I \times (x'-w, x'+w)$; we omit using/explaining higher regularity of $u$.

For $t \in \tilde I$ we set
\beq
\bar u(t) = \int_{x'-w}^{x'+w}u(t,x)dx, \quad \mu (t)= \int_{x'-w}^{x'+w}\left<\tilde \nu, \nu\right>dx
\eeq
and estimate for $t', t'' \in \tilde I$, $t'<t''$, that
\bea \label{speed_estimate2}
\bar u(t'')-\bar u( t') =&  \int_{t'}^{t''}\int_{x'-w}^{x'+w} \frac{\partial u}{\partial t}(t,x)dxdt \\
=&   \int_{t'}^{t''}\mu(t)  \frac{1}{\mu(t)}\int_{x'-w}^{x'+w} -\log\left(\frac{F}{f}\right)\left<\nu, \tilde \nu \right>dxdt \\
\ge&  - \int_{t'}^{t''}\mu(t)  \log\left(\frac{1}{\mu(t)}\int_{x'-w}^{x'+w} \frac{F}{f} \left<\nu, \tilde \nu \right>dx\right)dt \\
\ge&  - \int_{t'}^{t''}\mu(t)  \log\left(\frac{1}{\mu(t)}\int_{x'-w}^{x'+w} \frac{F(1, 1)}{2}\frac{H}{f} \left<\nu, \tilde \nu \right>dx\right)dt
\eea
where we used Jensen's inequality in the last but one inequality.
By assuming that $\Lambda_0$ is sufficiently large we may assume that $\delta$ from above in \eqref{normals_similar} is arbitrarily small and that, furthermore,
\beq \label{label_smallness}
|\mu(t)-2w|+|f(p)-f(0,0,1)| +|u(t,x)-\frac{1}{2w}\bar u(t)| \le \delta
\eeq
for all $t \in \tilde I$, $ p \in R(t, x')$ and  $ x\in (x'-w, x'+w)$.
Since $R(t, x')$ is part of a surface of revolution we can  rewrite the mean curvature $H$ of $R(t, x')$ by using this symmetry as follows (this well-known representation can be found e.g. in \cite[Equations (1)-(4)]{Ken}).
We parametrize for $t \in \tilde I$ the following piece of a strictly convex curve 
\beq
\gamma(t)= \{z=0\}\cap R(t, x') \cap \{y\ge 0\}
\eeq
 (in the $(x,y)$-plane) with respect to arc length  
\beq
\gamma(t)=\gamma(t,\theta)=(x(t,\theta), y(t,\theta), 0)
\eeq
with
$\theta \in [0, \delta(t)]$, $\delta(t)>0$ suitable,
so that 
\beq
x(t,0)=x'-w \le x(t,\delta(t))=x'+w, \quad \left\{\frac{\partial x}{\partial \theta}(t, \theta)\right\}^2+\left\{\frac{\partial y}{\partial \theta}(t, \theta)\right\}^2=1,
\eeq
then we have for the mean curvature $H(\gamma(t, \theta))$ of $M(t)$ in $\gamma(t, \theta)$ that
\beq
H(t, \theta):= H(\gamma(t, \theta))=
\frac{1-\frac{\partial}{\partial \theta}\left(y(t, \theta)\frac{\partial y}{\partial \theta}(t, \theta)\right)}{y(\theta)\frac{\partial x}{\partial \theta}(t, \theta)}.
\eeq
For every sufficiently large choice of $\Lambda_0$ we may assume that the tangent vector of $\gamma(t, \cdot)$
is approximately parallel to the $x$-axis, meaning that
\beq
\left|\frac{\partial y}{\partial \theta}(t,\theta)\right|\le \delta,  \quad \left|1-\frac{\partial x}{\partial \theta}(t,\theta)\right|\le \delta
\eeq
for all $t\in \tilde I$, $\theta \in [0, \delta(t)]$. Hence we have for $t\in \tilde I$ that
\beq
\int_0^{\delta(t)}H(t, \theta)d\theta \le \int_0^{\delta(t)}\frac{1}{y(t, \theta)(1-\delta)}d\theta+c\delta
\eeq
with a constant $c$ which holds uniformly for all $t \in \tilde I$.
Since the volume elements corresponding to the graphical and the arc  length parametrization of $\gamma(t, \cdot)$
can be estimated against each other by a factor in $(1-\delta, 1+\delta)$, $\Lambda_0>0$ large assumed, 
we conclude that for all $t\in \tilde I$ we have
\bea
\frac{F(1,1)}{2\mu(t)}\int_{x'-w}^{x'+w}\frac{H}{f}\left<\nu, \tilde \nu\right>dx \le& \frac{F(1,1)}{2\mu(t)} \int_0^{\delta(t)}\frac{H(t, \theta)}{f}d\theta(1+\delta) \\
\le&  \frac{F(1, 1)}{2\mu(t)} \frac{1+\delta}{f(0,0,1)-\delta} \int_0^{\delta(t)}H(t, \theta) d\theta \\
\le&  \frac{F(1, 1)}{2\mu(t)} \frac{1+\delta}{(1-\delta)(f(0,0,1)-\delta)}\int_0^{\delta(t)}\frac{1}{y(t, \theta)}d\theta+c\delta \\
\le&  \frac{1}{\mu(t)} \frac{1+\delta}{(1-\delta)(f(0,0,1)-\delta)}\int_0^{\delta(t)}\frac{f(0,0,1)}{1+\frac{\bar \delta}{8}}d\theta+c\delta \\
\le& \frac{1+c \delta}{1+\frac{\bar \delta}{8}}.
\eea
where $\delta=\delta(\bar \delta, d_0, \Lambda_0)>0$ is sufficiently small which is possible for $\Lambda_0$ sufficiently large and
where the last two inequalities are  valid (only) for $t \in I_{\varepsilon_0}=[\tilde t, \tilde t+\varepsilon_0)$  at the moment where $0<\varepsilon_0<t_1-\tilde t$ is chosen as large as possible so that
\beq \label{lower_b_diam}
\diam C(t) \ge \frac{2+\frac{\bar \delta}{4}}{f(0,0,1)}\frac{F(1, .., 1)}{2} 
\eeq
for $t\in I_{\varepsilon_0}$. Inserting this in \eqref{speed_estimate2} and using a first order Taylor's expansion of the logarithm around 1 gives summarized that
there is a small $c_0=c_0( d_0, \Lambda_0)>0$ such that
\beq \label{stern}
\bar u(t'') \ge \bar u(t')+(t''-t')c_0\bar \delta
\eeq
for all $t', t'' \in I_{\varepsilon_0}$ with $t'<t''$.
On the other hand $t \mapsto \bar u(t)$ is continuously differentiable, hence by taking the limit of the difference quotient we get
\beq
\frac{d}{dt}\bar u(t')= \lim_{t''\rightarrow t', t''>t'}\frac{\bar u(t'')-\bar u(t')}{t''-t'} \ge c_0 \bar \delta.
\eeq
Since $\bar u(t)$ is monotone increasing for $t\ge \tilde t$ with slope at least $c_0\bar \delta$ as long as \eqref{lower_b_diam} holds
and in view of our assumed contraposition to the claim of the lemma as well as \eqref{label_smallness} we deduce that
\beq
\frac{d}{dt}\bar  u(t) \ge c_0\bar \delta
\eeq
for all $t\in \tilde I$ and hence
\beq
\bar u(t) \ge \bar u(\tilde t) + (t-\tilde t)c_0 \bar \delta
\eeq
for all $t \in \tilde I$.
 At a sufficiently large time
$\bar t\in [\tilde t, t_1]$ the surface $M(\bar t)$ has large extension perpendicular to the $x$-axis and in the direction of the $x$-axis.
Note that we use here Assumption \ref{ass_long_extension}.
Hence, by convexity,  $M(\bar t)$ encloses a large sphere, a contradiction to the choice of $\Theta^{*}$.
\end{proof}

We continue the proof of Lemma \ref{extra_lemma}. Similarly as in the proof of Lemma \ref{claim2} we construct an outer barrier for $M(t)$, $t \in [t_0, t_1]$, consisting of two spheres and a cylinder. This time  the adjustment of the barrier happens by using Lemma \ref{claim3} in a more refined way, namely so that the diameter of the barrier is even sufficiently fast decreasing for our purposes (meaning, that we get a contradiction to the assumed contraposition \eqref{contraposition}). 

For it we fix $\bar \delta>0$ small according to Lemma \ref{claim3}
where we assume in order to realize it that $\Lambda_0>0$ is sufficiently large.
In order to specify the barrier we define the evolution of a cylinder $C(t)$ and two spheres $S_1(t)$ and $S_2(t)$ for  $t \in [t_0, t_1]$ as follows. We set
\bea \label{initial_value_barrier}
C(t_0)=&\left\{(x,y,z):\mathbb{R}^3:y^2+z^2=\left(\frac{F(1, ..., 1)}{2}\frac{1+\frac{3}{4}\bar \delta}{f(0,0,1)}\right)^2\right\} \\
S_1(t_0)=&\partial B_{\frac{F(1, ..., 1)(1+\bar \delta)}{2f(0,0,1)}}(p(t_0), 0, 0) \\
S_2(t_0)=&\partial B_{\frac{F(1, ..., 1)(1+\bar \delta)}{2f(0,0,1)}}(q(t_0), 0, 0),
\eea
cf. \eqref{anfangende}.
Clearly, we have by construction
\beq
M(t_0) \subset \inte \left\{\co C(t_0) \cap \co (S_1(t_0)\cup S_2(t_0))\right\}.
\eeq
By assumption (\ref{ass_for_f}) we have that
\begin{equation} 
f(0,0,1)\ge (1+\tilde \delta) \frac{f}{c(n)}.
\end{equation}
For a small period of time $\Delta t>0$ we define the flows $C(t)$, $S_1(t)$ and $S_2(t)$ for $t \in [t_0, t_0+\Delta t]$
by prescribing that the three surfaces flow according to \eqref{initial_value_barrier}
and Equation \eqref{flow1} for these $t$.
For $\Delta t=\Delta t(\bar \delta)$ sufficiently small, the flows $S_i(t)$, $i=1,2$, contract at a positive speed inward directed. This follows from the continuity of the flows and the evaluations of their speeds at time $t=t_0$, for it note that
\begin{equation} \label{speed_estimate}
\begin{aligned}
\frac{F}{f}=&\frac{2f(0,0,1)}{(1+\bar \delta)f} \\
>& \frac{1+\tilde \delta}{1+\bar \delta} \\
>& 1,
\end{aligned}
\end{equation} 
provided $\bar \delta>0$ is sufficiently small. The latter can be achieved by assuming that $\Lambda_0=\Lambda_0(f)$ is sufficiently large. Note that $\tilde \delta$ depends solely on $f$.
Choosing $\Delta t$ possibly smaller and after this $\eta=\eta(\Delta t)>0$ sufficiently small we may assume that for 
$t\in [t_0, t_0+\Delta t]$ the $C^1$-distance between $S_1(t)$ and $S_1(0)$ is not too large
and that on the other hand 
\beq \label{99999}
S_1(t_0+\Delta t) \subset \co \partial B_{\frac{F(1,1)(1+\bar \delta)}{2f(0,0,1)}-\eta}(p(t_0), 0,0)
\eeq
which is a kind of quantified minimum improvement of the barrier.
By symmetry, the analogous statement is true for $S_2(t)$.
For $t\in [t_0, t_0+\Delta t]$ the surfaces $C(t)$ slowly expand, since 
\beq
\frac{H}{f} <1
\eeq
on $C(t_0)$. Choosing $\Delta t>0$ possibly smaller we may assume that
\beq \label{C(t)_is_contained}
C(t) \cap \left\{x=\frac{p(t_0)+q(t_0)}{2}\right\}\subset \inte \co (S_1(t) \cup S_2(t))
\eeq 
for all $t \in [t_0, t_0+\Delta t]$
since that is true for $t=t_0$ by construction. In other words the radii of the evolving cylinder and sphere
approach each other but the strict inequality between them which holds in $t=t_0$  is preserved.

Now the following barrier property holds, namely
\beq \label{M(t)_is_contained}
M(t) \subset \co (S_1(t) \cup S_2(t)), \quad M(t) \subset \co C(t)
\eeq 
for $t\in [t_0, t_0+\Delta t]$. 
This can be seen by deducing from the contraposition of \eqref{M(t)_is_contained} via Lemma \ref{maximum_principle} a contradiction, similarly, as we already did this in the proof of Lemma \ref{claim2}.

Note that \eqref{beh1} gives for $t\in [t_0, t_0+\Delta t]$ a 'better estimate' for $M(t)$ than the barrier $C(t)$ since the latter expands during that time. 

While the evolution of $M(t)$ is according to the PDE \eqref{flow1} for all $t\in [t_0, t_1]$, we stop the flows
$C(t)$, $S_1(t)$ and $S_2(t)$ at time $t=t_0+\Delta t$ and adjust them by definition as follows.
We set
\bea
C(t_0+\Delta t)=&\left\{(x,y,z):\mathbb{R}^3:y^2+z^2=\left(\frac{F(1,1)(1+\frac{3}{4}\bar \delta)}{2f(0,0,1)}\right)^2\right\} \\
S_1(t_0+\Delta t)=&\partial B_{\frac{F(1,1)(1+\bar \delta)}{2f(0,0,1)}}(p(t_0)+\eta/2, 0, 0) \\
S_2(t_0+\Delta t)=&\partial B_{\frac{F(1,1)(1+\bar \delta)}{2f(0,0,1)}}(q(t_0)-\eta/2, 0, 0).
\eea
Clearly, we have by construction
\beq
M(t_0+\Delta t) \subset \inte \left\{\co C(t_0+\Delta t) \cap \co (S_1(t_0+\Delta t)\cup S_2(t_0+\Delta t))\right\}
\eeq
in view of \eqref{99999} and \eqref{beh1}.
Note that 
\beq
\diam  \co (S_1(t_0+\Delta t)\cup S_2(t_0+\Delta t)) \le \diam  \co (S_1(t_0)\cup S_2(t_0))-\eta.
\eeq
Iteration of this procedure leads to a diameter estimate for the $M(t)$ of the form
\begin{equation} \label{desired_decay}
d(t_0+k\Delta t) \le d(t_0)+c(f)-k\eta 
\end{equation} 
with a constant 
\beq
c(f)=\frac{F(1,1)(1+ \bar \delta )}{f(0,0,1)}
\eeq
and $k\in \mathbb{N}$ such that
\beq
t_0+k \Delta t \le t_1.
\eeq
But this contradicts Assumption \ref{ass_long_extension} Item (1).
 
The proof of Lemma \ref{extra_lemma} is complete.
\end{proof}

\section*{Acknowledgment}
During the preparation of this work the author has been funded temporarily  by the Deutsche Forschungsgemeinschaft (DFG,
German Research Foundation) project: 319506420 and while revising the paper
by the Deutsche Forschungsgemeinschaft (DFG,
German Research Foundation) project: 404870139.
 We thank the anonymous referee a lot for reading our first version which lacked many details and was elaborated-to-read.
 Following the referee's suggestion we added several more details and explanations.
 In the course of the revision we found additionally some minor technical improvements which lead to a more general main result
while all arguments were basically already contained in our original version.

\end{document}